\documentclass[11pt]{amsart}
\usepackage{amsmath,amssymb,amsthm}
\usepackage[all]{xy}
\usepackage{comment}

\newtheorem{theorem}{Theorem}[section]
\newtheorem{lemma}[theorem]{Lemma}
\newtheorem{proposition}[theorem]{Proposition}

\theoremstyle{definition}
\newtheorem{definition}[theorem]{Definition}
\newtheorem{remark}[theorem]{Remark}

\numberwithin{equation}{section}

\renewcommand{\phi}{\varphi}

\newcommand{\ep}{\varepsilon}

\newcommand{\Aut}{\operatorname{Aut}}
\newcommand{\Aff}{\operatorname{Aff}}
\newcommand{\Bott}{\operatorname{Bott}}
\newcommand{\Coker}{\operatorname{Coker}}
\newcommand{\Hom}{\operatorname{Hom}}
\newcommand{\Homeo}{\operatorname{Homeo}}
\newcommand{\id}{\operatorname{id}}
\newcommand{\Ima}{\operatorname{Im}}
\newcommand{\Iso}{\operatorname{Iso}}
\newcommand{\Ker}{\operatorname{Ker}}
\newcommand{\sgn}{\operatorname{sgn}}

\newcommand{\N}{\mathbb{N}}
\newcommand{\Z}{\mathbb{Z}}
\newcommand{\Q}{\mathbb{Q}}
\newcommand{\R}{\mathbb{R}}

\newcommand{\T}{\mathbb{T}}

\newcommand{\K}{\mathbb{K}}

\newcommand{\cG}{\mathcal{G}}
\newcommand{\cH}{\mathcal{H}}
\newcommand{\cM}{\mathcal{M}}
\newcommand{\cO}{\mathcal{O}}

\newcommand{\sfb}{\mathsf{b}}
\newcommand{\sF}{\mathsf{F}}

\title{HK and GL}
\author{Hiroki Matui}
\address{Graduate School of Mathematical Sciences, 
The University of Tokyo, Japan}
\email{hiroki@ms.u-tokyo.ac.jp}
\thanks{The author was supported by JSPS KAKENHI Grant Number 23K22397.}
\date{}

\begin{document}

\begin{abstract}
We study the HK conjecture and the gap-labelling problem 
for transformation groupoids associated with 
free actions of poly-$\Z$ groups on Cantor sets. 
The main tool is a comparison of 
the long exact sequences in groupoid homology and cohomology 
with the Pimsner--Voiculescu exact sequence 
for crossed products by $\Z$. 
In addition to the canonical homology comparison maps 
$\mu_0$ and $\mu_1$, 
we introduce cohomology comparison maps 
associated with suitable $K$-theory classes of the acting group. 
Together with Poincar\'e duality, 
these maps detect the higher homology terms 
occurring in the HK conjecture. 

We apply this method to free actions of poly-$\Z$ groups of 
small Hirsch length. 
For actions of $\Z$, $\Z^2$, and the Klein bottle group, 
we recover HK and gap-labelling. 
For several classes of groups of Hirsch length three and four, 
we either prove HK or obtain explicit exact sequences 
describing the $K$-groups 
in terms of groupoid homology and cohomology. 

For gap-labelling, 
we combine the de la Harpe--Skandalis determinant, 
the trace formula for the Pimsner--Voiculescu boundary map, and 
transposition decompositions in topological full groups. 
This gives gap-labelling up to a factor of two 
for all free actions of poly-$\Z$ groups of Hirsch length three 
and for certain groups of Hirsch length four, including $\Z^4$. 
We also recover gap-labelling for $\Z^3$-actions and 
prove gap-labelling up to a factor of two for $\Z^5$-actions 
by using cohomology comparison maps for mapping tori. 
\end{abstract}

\maketitle

\section{Introduction}

\subsection{What is HK?\nopunct}

A homology theory for ample groupoids was introduced 
in \cite{Ma12PLMS} and has since been used 
in the study of totally disconnected dynamical systems, 
topological full groups, and groupoid $C^*$-algebras. 
In this framework, the HK conjecture predicts 
a parity-preserving relationship between the $K$-theory of 
the reduced groupoid $C^*$-algebra and the homology groups of 
the groupoid: 
\[
K_i(C^*_r(\cG))\cong\bigoplus_{j\geq 0}H_{2j+i}(\cG),\qquad i=0,1. 
\]
An early indication of such a relationship can already be found 
in the work of Forrest and Hunton on $\Z^N$-actions 
on the Cantor set \cite{FH99ETDS}. 
The integral picture is, however, more delicate 
than one might expect from this early work, 
since torsion phenomena may occur. 

The conjecture was verified for groupoids 
arising from one-sided shifts of finite type in \cite{Ma15crelle}, 
and for products of such groupoids in \cite{Ma16Adv}. 
Since then, a number of positive results have been obtained 
in various directions 
\cite{FKPS19Munster,Li15JFA,Or20JNG,Yi20BullAust,BDGW23PLMS}. 
The works \cite{PY22ETDS,PY23JNG,PY2509arXiv} of 
Proietti and Yamashita have also clarified 
the conceptual background of the conjecture, 
relating groupoid homology, $K$-theory, and Chern characters 
in a systematic way. 
In particular, 
they established the rational form of the HK conjecture 
under general and mild assumptions. 

On the other hand, the HK conjecture is not valid 
for arbitrary ample groupoids. 
Scarparo gave counterexamples arising from odometers 
\cite{Sc20ETDS}, and 
Deeley later constructed a counterexample 
which is principal \cite{De23ETDS}. 
Further examples and related phenomena have been studied, 
for instance, in \cite{OS22MathScand,OS23GGD}. 
Thus the present state of the subject is rather subtle: 
the conjecture holds in many natural low-dimensional or 
highly structured situations, 
but fails in full generality. 
In particular, even for transformation groupoids of 
$\Z^N$-actions, 
the integral form of the conjecture is 
far from completely understood when $N\geq3$. 
To the best of our knowledge, 
no explicit verification of the integral HK conjecture 
for free Cantor $\Z^3$-systems had previously appeared.

\subsection{What is GL?\nopunct}

The gap-labelling conjecture originates in Bellissard's 
noncommutative-geometric approach to Schr\"{o}dinger operators 
on aperiodic media. 
The integrated density of states is constant on each spectral gap, 
and the problem is to describe the resulting gap labels 
in terms of the topology and dynamics of the underlying aperiodic structure. 
In the $C^*$-algebraic formulation, 
for a Cantor $\Z^N$-system $\Z^N\curvearrowright X$ and 
an invariant probability measure $\nu$, 
the conjecture takes the form 
\[
\tau_{\nu,*}\left(K_0(C(X)\rtimes\Z^N)\right)=\nu(C(X,\Z)), 
\]
where $\tau_\nu$ is the trace induced by $\nu$. 
This formulation is closely related to the gap-labelling theorems 
for aperiodic solids and tiling spaces 
developed by Bellissard and collaborators \cite{MR1798994}. 

Around the beginning of the 2000s, several proofs, 
or closely related formulations, appeared in 
\cite{KP03Michigan,MR2018220,BBG06CMP}. 
They were based on index theory and on comparisons 
between $K$-theory and cohomological data of suspensions or tiling spaces. 
The integral statement is, however, more delicate than 
these proofs might suggest. 
G\"{a}hler, Hunton and Kellendonk observed in \cite{GHK13AGT} 
that torsion in integral cohomology may invalidate 
torsion-freeness assumptions implicit in earlier arguments. 
From the present point of view, 
the essential issue is that these arguments require 
an integral Chern-character identification 
which is not justified for arbitrary Cantor $\Z^N$-systems. 

More recently, Liu, Rosenberg and Trevi{\~n}o constructed a
four-dimensional cubical substitution tiling space 
for which the relevant Chern character fails to give 
the required integral identification \cite{MR5056667}. 
The discrepancy in their example arises from torsion and 
is therefore annihilated by the Ruelle--Sullivan current, 
so it does not yield a counterexample to 
the original gap-labelling conjecture. 
Thus, the general integral problem should not be regarded 
as completely settled, 
although the cases $N=1$ and $N=2$ follow 
by elementary $K$-theoretic methods and 
the case $N=3$ was proved by Bellissard, Kellendonk and Legrand 
\cite{BKL01CRASParis}. 

In this paper, we formulate GL in the setting of ample groupoids. 
Let $\cG$ be an ample groupoid with compact unit space and 
let $\nu\in M(\cG)$. 
The measure $\nu$ induces a trace $\tau_\nu$ on $C_r^*(\cG)$, 
while integration descends to $H_0(\cG)$. 
The groupoid version of GL asks whether 
\[
\tau_{\nu,*}\left(K_0(C_r^*(\cG))\right)=\nu(H_0(\cG)). 
\]
This question was already raised, 
in a slightly more restricted setting, in \cite[Remark 3.8]{Ma16Adv}. 
For transformation groupoids $X\rtimes\Z^N$, 
this recovers the classical Cantor-dynamical formulation.

\subsection{Main results and methods}

The purpose of the present paper is 
to study HK and GL for transformation groupoids 
arising from free actions of poly-$\Z$ groups on Cantor sets. 
Our emphasis is not only on obtaining positive results 
in particular cases, 
but also on developing a concrete method 
that compares groupoid homology, groupoid cohomology, and 
the $K$-theory of the associated crossed product $C^*$-algebra. 

The starting point of our approach is the observation that, 
for semidirect products by $\Z$-actions, 
the homology long exact sequence and 
the Pimsner--Voiculescu exact sequence should be compared directly. 
We prove that the homology comparison maps
\[
\mu_0:H_0(\cG)\longrightarrow K_0(C^*_r(\cG)),\qquad 
\mu_1:H_1(\cG)\longrightarrow K_1(C^*_r(\cG))
\]
are compatible with these exact sequences. 
This gives a rather explicit way of 
tracing homology classes into $K$-theory. 
Although this argument is somewhat computational, 
it has the advantage that 
it keeps track of the actual comparison maps, 
rather than only identifying abstract groups. 

A new ingredient in this paper is 
the use of cohomology comparison maps. 
For a transformation groupoid $\cG=X\rtimes\Gamma$ and 
for a suitable class $\sfb\in K_i(C^*_r(\Gamma))$, 
we construct maps 
\[
\mu^0:H^0(\cG)\longrightarrow K_i(C^*_r(\cG)),\qquad 
\mu^1:H^1(\cG)\longrightarrow K_{i+1}(C^*_r(\cG)). 
\]
The map $\mu^0$ is obtained 
from invariant integer-valued functions and the class $\sfb$, 
while $\mu^1$ is defined using the circle actions on $C^*_r(\cG)$ 
associated with continuous homomorphisms $\cG\to\Z$. 
These maps are again compatible with 
the cohomology long exact sequence and 
the Pimsner--Voiculescu exact sequence. 
In this way, 
the missing $K$-theoretic summands 
which are not detected by $\mu_0$ and $\mu_1$ 
alone can be detected by cohomology. 
For actions of poly-$\Z$ groups, 
Poincar\'e duality then allows us to reinterpret 
these cohomological terms as the higher homology groups appearing in HK. 

Using this method, 
we obtain several low-dimensional results for free actions 
of poly-$\Z$ groups. 
For $\Z$-actions and $\Z^2$-actions, 
we recover HK and GL by elementary Pimsner--Voiculescu arguments. 
For actions of the Klein bottle group, 
we show that $\mu_1$ is an isomorphism and 
that $K_0(C^*_r(\cG))$ fits into a split exact sequence built 
from $H_0(\cG)$ and $H_2(\cG)$. 
In particular, HK and GL hold in this case. 
For poly-$\Z$ groups of Hirsch length three, 
we divide the discussion according to the action 
on the relevant Bott class. 
In the orientation-preserving case we prove HK, 
while in the other cases we obtain explicit exact sequences 
which show, in particular, that
the canonical maps $\mu_0$ and $\mu_1$ are injective. 
We also treat certain poly-$\Z$ groups of Hirsch length four 
and obtain explicit exact sequences describing the $K$-groups 
in terms of the homological and cohomological data. 
These results do not amount to a complete verification of HK 
in all cases, 
but they give a precise description of 
how the $K$-groups are assembled from the homological data. 
In particular, they single out the injectivity of 
the canonical homology comparison maps as a natural weaker form of HK. 

We next turn to GL. 
Our formulation of GL is purely groupoid-theoretic. 
If $\cG$ is an ample groupoid with compact unit space, 
we consider the dimension map 
\[
D_\cG:K_0(C^*_r(\cG))\longrightarrow\Aff(M(\cG))
\]
defined by pairing $K_0(C^*_r(\cG))$ with invariant probability measures. 
Then GL asks whether the image of $D_\cG$ is 
already generated by the image of $H_0(\cG)$ 
under the canonical map $\mu_0$. 
In this form, the gap-labelling problem becomes 
a question about the trace values of 
$K$-theory classes and their relation to groupoid homology. 

A key technical point is that, 
even when full GL is difficult to prove, 
the obstruction can often be controlled up to a factor of two. 
The reason for this factor is quite concrete. 
We use the structure of topological full groups, 
in particular the fact that elements in the kernel of the index map 
can be written as products of transpositions. 
Combined with the de la Harpe--Skandalis determinant and 
the trace formula for the Pimsner--Voiculescu boundary map, 
this shows that certain determinant contributions lie in 
\[
\frac{1}{2}\Ima(D_\cG\circ\mu_0). 
\]
This leads naturally to what we call GL up to a factor of two: 
\[
\Ima D_\cG\subset\frac{1}{2}\Ima(D_\cG\circ\mu_0). 
\]
Thus the factor $1/2$ is not merely an artifact of notation, 
but arises from the transposition decomposition 
in the topological full group. 

Applying this argument, 
we prove GL up to a factor of two 
for free actions of poly-$\Z$ groups of Hirsch length three. 
We also prove it for certain groups of Hirsch length four, 
including the case of $\Z^4$. 
In these cases, full GL follows 
whenever $\Ima(D_\cG\circ\mu_0)$ is $2$-divisible. 
Although our method does not settle 
integral GL in complete generality, 
it identifies the precise place where the remaining obstruction occurs. 

Finally, 
we revisit the three-dimensional gap-labelling theorem 
of \cite{BKL01CRASParis} 
from the viewpoint of mapping tori and cohomology comparison maps. 
For a free $\Z^N$-action, 
the mapping torus provides another realization of 
the relevant $K$-theory groups. 
We show that the cohomology comparison maps for mapping tori are 
again compatible with the exact sequences. 
When $N$ is odd, the antisymmetry in the Connes pairing forces 
certain cohomological contributions to have zero trace. 
This recovers GL for $\Z^3$-actions and, 
combined with the preceding argument, 
yields GL up to a factor of two for $\Z^5$-actions. 

In summary, the paper develops an exact-sequence method 
for studying HK and GL simultaneously. 
The method gives explicit information 
on the comparison between homology and $K$-theory 
for low-dimensional poly-$\Z$ actions, 
introduces cohomology comparison maps 
as a systematic complement to the canonical homology comparison maps, 
and proves several gap-labelling results, 
including gap-labelling up to a factor of two 
where the full integral statement is not presently accessible 
by known methods.

\section*{Acknowledgements}

The author is deeply grateful to Valerio Proietti and Alistair Miller 
for pointing out difficulties in the existing proofs of 
the gap-labelling theorem. 
Their comments, made at an early stage, 
provided the starting point for the present work. 

The author would also like to thank Johannes Kellendonk 
for many helpful comments and explanations. 
In particular, he explained the vanishing of the relevant Connes pairing 
in the three-dimensional gap-labelling theorem, 
drew the author's attention to the recent work of 
Liu, Rosenberg and Trevi{\~n}o, and 
offered valuable encouragement concerning this paper. 

Finally, the author thanks Ian Putnam 
for candid and helpful correspondence concerning 
the proof of the gap-labelling theorem in his joint work with Kaminker.

\section{Preliminaries}

\subsection{Ample groupoids}

We identify $\T$ with $\R/\Z$. 
The indicator function of a set $A$ is denoted by $1_A$. 
We say that a subset of a topological space $X$ is clopen 
if it is both closed and open. 
A topological space is said to be totally disconnected 
if its connected components are singletons. 
By a Cantor set, 
we mean a compact, metrizable, totally disconnected space 
with no isolated points. 
It is known that any two such spaces are homeomorphic. 

In this article, by an \'etale groupoid 
we mean a second countable locally compact Hausdorff groupoid 
such that the range map is a local homeomorphism. 
We refer the reader to \cite{Re_text,R08Irish} 
for background material on \'etale groupoids. 
For an \'etale groupoid $\cG$, 
we let $\cG^{(0)}$ denote the unit space and 
let $s$ and $r$ denote the source and range maps, 
i.e.,\ $s(g)=g^{-1}g$, $r(g)=gg^{-1}$. 
A subset $U\subset\cG$ is called a bisection 
if $r|U$ and $s|U$ are injective. 
The \'etale groupoid $\cG$ has a basis for its topology 
consisting of open bisections. 
For $x\in\cG^{(0)}$, 
the set $r(s^{-1}(x))$ is called the orbit 
(or $\cG$-orbit) of $x$. 
When every orbit is dense in $\cG^{(0)}$, $\cG$ is said to be minimal. 
An \'etale groupoid $\cG$ is called ample 
if it has a basis consisting of compact open bisections. 
Under our standing assumptions, 
this is equivalent to $\cG^{(0)}$ being totally disconnected. 
For $x\in\cG^{(0)}$, 
we write $\cG_x=r^{-1}(x)\cap s^{-1}(x)$ and 
call it the isotropy group of $x$. 
The isotropy subgroupoid of $\cG$ is 
$\Iso(\cG):=\{g\in\cG\mid r(g)=s(g)\}=\bigsqcup_{x\in\cG^{(0)}}\cG_x$. 
We say that $\cG$ is principal if $\Iso(\cG)=\cG^{(0)}$. 
We say that $\cG$ is effective 
if the interior of $\Iso(\cG)$ is $\cG^{(0)}$. 
This condition is also called essential principality 
in much of the literature. 
A (non-zero) Borel measure $\nu$ on $\cG^{(0)}$ is 
said to be $\cG$-invariant 
if for every compact open bisection $U\subset\cG$, 
one has $\nu(r(U))=\nu(s(U))<\infty$. 
When $\cG^{(0)}$ is compact, 
we let $M(\cG)$ denote the space of 
$\cG$-invariant probability measures on $\cG^{(0)}$. 

Let $\cG$ be an ample groupoid with compact unit space. 
We say that a compact open bisection $U\subset\cG$ is full 
if $r(U)=\cG^{(0)}=s(U)$ holds. 
Full bisections form a group $\sF(\cG)$ with multiplication given by 
\[
UV:=\{gh\mid g\in U,\ h\in V,\ s(g)=r(h)\}. 
\]
We call $\sF(\cG)$ the topological full group of $\cG$. 
Each full bisection $U$ induces 
a homeomorphism $r\circ(s|U)^{-1}$ of $\cG^{(0)}$. 
When $\cG$ is effective, the resulting homomorphism 
$\sF(\cG)\to\Homeo(\cG^{(0)})$ is injective (\cite{Ma12PLMS}). 

From a group action on a topological space, 
we can form a transformation groupoid. 
Let $\Gamma\curvearrowright X$ be an action of 
a countable discrete group $\Gamma$ 
on a second countable locally compact Hausdorff space $X$. 
The transformation groupoid $\cG:=X\rtimes\Gamma$ is 
$X\times\Gamma$ equipped with the product topology. 
The unit space of $\cG$ is given by $\cG^{(0)}=X\times\{1\}$ 
(where $1$ is the identity of $\Gamma$), 
with range and source maps 
$r(x,\gamma)=(x,1)$ and $s(x,\gamma)=(\gamma^{-1}.x,1)$. 
The unit space $\cG^{(0)}$ is often identified with $X$. 
Multiplication is given by
\[
(x,\gamma)\cdot(\gamma^{-1}.x,\gamma')=(x,\gamma\gamma'), 
\]
and the inverse of $(x,\gamma)$ is $(\gamma^{-1}.x,\gamma^{-1})$. 
Such a transformation groupoid is always \'etale. 
It is ample if and only if $X$ is totally disconnected. 
Moreover, $\cG$ is minimal if and only if the action is minimal 
that is, for all $x\in X$, 
the orbit $\{\gamma.x\mid\gamma\in\Gamma\}$ is dense in $X$. 
A Borel measure $\nu$ on $\cG^{(0)}$ is $\cG$-invariant 
if and only if $\nu$ is $\Gamma$-invariant 
under the identification of $\cG^{(0)}$ and $X$.

\subsection{Homology and cohomology for ample groupoids}

In this subsection, 
we collect several facts concerning the homology and cohomology 
of ample groupoids which will be used in the subsequent sections. 

Homology theory for \'etale groupoids was introduced 
by Crainic and Moerdijk \cite{CM00crelle}, 
and reframed by the author in \cite{Ma12PLMS} for ample groupoids. 
The reader may also refer to \cite{Ma16Adv,BDGW23PLMS,Ma25JLMS}. 
When $\cG$ is the transformation groupoid of 
an action $\Gamma\curvearrowright X$ of a discrete group $\Gamma$, 
the homology groups $H_n(\cG)$ are canonically isomorphic to 
the usual group homology $H_n(\Gamma,C_c(X,\Z))$. 
Here $C_c(X,\Z)$ is regarded as a $\Gamma$-module 
via the action induced from $\Gamma\curvearrowright X$. 
The zeroth homology $H_0(\cG)$ is the group of coinvariants: 
\[
H_0(\cG)=C_c(\cG^{(0)},\Z)
/\langle1_{r(U)}-1_{s(U)}
\mid\text{$U\subset\cG$ is a compact open bisection}\rangle. 
\]
There exists a natural homomorphism 
$\mu_0:H_0(\cG)\to K_0(C^*_r(\cG))$ such that 
the following diagram commutes: 
\[
\xymatrix@M=8pt{
C_c(\cG^{(0)},\Z) \ar[r]^-{\cong} \ar[d] & 
K_0(C_0(\cG^{(0)})) \ar[d] \\
H_0(\cG) \ar[r]^-{\mu_0} & 
K_0(C^*_r(\cG)). 
}
\]
When $U\subset\cG$ is a compact open bisection 
satisfying $r(U)=s(U)$, 	
the function $1_U$ is a $1$-chain and 
yields an element $[1_U]\in H_1(\cG)$. 
If $\cG^{(0)}$ is compact and $r(U)=\cG^{(0)}=s(U)$, 
the function $1_U$ can also be regarded 
as a unitary in the unital $C^*$-algebra $C^*_r(\cG)$ 
which defines the $K_1$-class $K_1(1_U)\in K_1(C^*_r(\cG))$. 
In \cite{BDGW23PLMS}, 
it was shown that there exists a natural homomorphism 
$\mu_1:H_1(\cG)\to K_1(C^*_r(\cG))$ 
such that $\mu_1([1_U])=K_1(1_U)$ holds. 
In other words, this means that the following diagram commutes: 
\[
\xymatrix@M=8pt{
\sF(\cG) \ar[r] \ar[d]_-{I} & 
U(C^*_r(\cG)) \ar[d] \\
H_1(\cG) \ar[r]^-{\mu_1} & 
K_1(C^*_r(\cG)), 
}
\]
where $\sF(\cG)$ is the topological full group, 
the top horizontal arrow sends $U$ to $1_U$ and 
$I:\sF(\cG)\to H_1(\cG)$ is the index map (\cite{Ma12PLMS}) 
defined by $U\mapsto [1_U]$. 
We call $\mu_0$ and $\mu_1$ the homology comparison maps. 

Let us recall the long exact sequence 
for the homology of semidirect products by $\Z$-actions. 

\begin{lemma}[{\cite[Lemma 1.3]{Or20JNG}}]
Let $\cG$ be an ample groupoid and 
let $\xi:\cG\to\Z$ be a homomorphism. 
Let $\hat\xi\in\Aut(\cG\times_\xi\Z)$ be an automorphism 
of the skew product $\cG\times_\xi\Z$ 
defined by $\hat\xi(g,i):=(g,i+1)$. 
Then, there exists a long exact sequence 
\[
\xymatrix@M=8pt{
\cdots \ar[r] & 
H_{n+1}(\cG) \ar[r] & 
H_n(\cG\times_\xi\Z) \ar[r]^{\id-H_n(\hat\xi)} & 
H_n(\cG\times_\xi\Z) \ar[r] & 
H_n(\cG) \ar[r] & \cdots . 
}
\]
\end{lemma}

\begin{lemma}\label{hoLES}
Let $\cH$ be an ample groupoid 
and let $\theta:\cH\to\cH$ be an automorphism. 
Then, there exists a long exact sequence 
\[
\xymatrix@M=8pt{
\cdots \ar[r] & 
H_{n+1}(\cH\rtimes_\theta\Z) \ar[r] & 
H_n(\cH) \ar[r]^{\id-H_n(\theta)} & 
H_n(\cH) \ar[r] & 
H_n(\cH\rtimes_\theta\Z) \ar[r] & \cdots 
}
\]
where $\cH\rtimes_\theta\Z$ is the semidirect product. 
\end{lemma}

We use the convention that 
$\cH\rtimes_\theta\Z$ is $\cH\times\Z$ with 
\[
r(h,k):=(r(h),0),\qquad s(h,k):=(\theta^{-k}(s(h)),0), 
\]
and multiplication
\[
(h,k)\cdot(h',l):=(h\theta^k(h'),k+l), 
\]
whenever $\theta^{-k}(s(h))=r(h')$. 

\begin{proof}
Define a homomorphism $\xi:\cH\rtimes_\theta\Z\to\Z$ 
by $\xi(g,k):=k$. 
We apply the lemma above to $\cH\rtimes_\theta\Z$ and $\xi$. 
For any $x\in\cH^{(0)}$ and $k\in\Z$, 
$r(x,k,0)=(x,0,0)$ and $s(x,k,0)=(\theta^{-k}(x),0,k)$. 
Hence $\cH^{(0)}\times\{0\}\times\{0\}$ is full 
in $(\cH\rtimes_\theta\Z)\times_\xi\Z$. 
Since $\cH$ is isomorphic to 
the reduction of $(\cH\rtimes_\theta\Z)\times_\xi\Z$ 
to $\cH^{(0)}\times\{0\}\times\{0\}$, 
one has $H_*((\cH\rtimes_\theta\Z)\times_\xi\Z)\cong H_*(\cH)$. 
Under this identification, 
the reduction of $\hat\xi$ to $\cH^{(0)}\times\{0\}\times\{0\}$ 
is implemented by the compact open bisection 
$\cH^{(0)}\times\{1\}\times\{0\}$, 
and the resulting automorphism of the
reduction is precisely $\theta$. 
\end{proof}

Cohomology theory for topological groupoids 
goes back to 
\cite[Chapter III]{MR660658} or \cite[Definition 1.12]{Re_text}. 
The reader may also refer to \cite[Definition 2.4]{MM25Munster} 
or \cite[Definition 2.4]{Ma25JLMS} 
for the cohomology of ample groupoids with constant coefficients. 
As with homology, 
when $\cG$ is the transformation groupoid of 
an action $\Gamma\curvearrowright X$ of a discrete group $\Gamma$, 
the cohomology groups $H^n(\cG)$ are canonically isomorphic to 
the usual group cohomology $H^n(\Gamma,C(X,\Z))$. 
The zeroth cohomology $H^0(\cG)$ is the group of invariants: 
\[
H^0(\cG)=\{f\in C(\cG^{(0)},\Z)
\mid f(r(g))=f(s(g))\quad\forall g\in\cG\}. 
\]
We let $\Hom(\cG,\Z)$ denote the group of 
continuous homomorphisms from $\cG$ to $\Z$. 
The first cohomology $H^1(\cG)$ is 
$\Hom(\cG,\Z)$ modulo the subgroup consisting of 
homomorphisms of the form $g\mapsto f(r(g))-f(s(g))$ 
for some $f\in C(\cG^{(0)},\Z)$. 

Analogously to the homology case, 
there exists a long exact sequence 
for the cohomology of semidirect products by $\Z$-actions. 
We use the convention that 
$\theta$ acts on cohomology by pullback through $\theta^{-1}$. 

\begin{lemma}[{\cite[Corollary 2.7]{Ma25JLMS}}]\label{coLES}
Let $\cH$ be an ample groupoid 
and let $\theta:\cH\to\cH$ be an automorphism. 
Then, there exists a long exact sequence 
\[
\xymatrix@M=8pt{
\cdots \ar[r] & 
H^n(\cH\rtimes_\theta\Z) \ar[r] & 
H^n(\cH) \ar[r]^{\id-H^n(\theta)} & 
H^n(\cH) \ar[r] & 
H^{n+1}(\cH\rtimes_\theta\Z) \ar[r] & \cdots 
}
\]
where $\cH\rtimes_\theta\Z$ is the semidirect product. 
\end{lemma}

We need one more ingredient for later use. 
Let $\cG$ be an ample groupoid with $\cG^{(0)}$ compact. 
For $\xi\in\Hom(\cG,\Z)$, one can define 
an action $\alpha_\xi:\T\curvearrowright C^*_r(\cG)$ by 
\[
\alpha_{\xi,t}(a)(g):=e^{2\pi i\xi(g)t}a(g)
\quad \forall a\in C_c(\cG),\ g\in\cG. 
\]
This induces a homomorphism 
$\tilde\alpha_\xi:C^*_r(\cG)\to C(\T)\otimes C^*_r(\cG)$ 
given by $\tilde\alpha_\xi(a)(t):= \alpha_{\xi,t}(a)$. 
Let $j:C^*_r(\cG)\to C(\T)\otimes C^*_r(\cG)$ be 
the canonical embedding, i.e. $j(a):=1\otimes a$. 
Since the two homomorphisms agree 
after evaluation at the base point of $\T$, 
their difference is naturally regarded as an element of 
$KK(C^*_r(\cG),C_0(\R)\otimes C^*_r(\cG))$. 

\begin{lemma}
The assignment $\xi\mapsto KK(\tilde\alpha_\xi)-KK(j)$ gives rise 
to a well-defined homomorphism 
from $H^1(\cG)$ to $KK^1(C^*_r(\cG),C^*_r(\cG))$. 
\end{lemma}

\begin{proof}
Suppose that $\xi$ is a coboundary, 
i.e., there exists $f\in C(\cG^{(0)},\Z)$ 
such that $\xi(g)=f(r(g))-f(s(g))$ for all $g\in\cG$. 
Define a unitary $w\in C(\T)\otimes C(\cG^{(0)})$ 
by $w(t):=e^{2\pi itf}$. 
For $t\in\T$, $a\in C_c(\cG)$ and $g\in\cG$, it holds that 
\[
\alpha_{\xi,t}(a)(g)=e^{2\pi i\xi(g)t}a(g)
=e^{2\pi if(r(g))t}a(g)e^{-2\pi if(s(g))t}
=(w(t)aw(t)^*)(g). 
\]
Hence, $\tilde\alpha_\xi$ is unitarily equivalent to $j$, 
and so $KK(\tilde\alpha_\xi)-KK(j)=0$. 

Suppose that we are given $\xi,\eta\in\Hom(\cG,\Z)$. 
The loop $t\mapsto \alpha_{\xi+\eta,t}$ is homotopic, 
through loops of automorphisms, 
to the concatenation of the loops 
$t\mapsto \alpha_{\xi,t}$ and $t\mapsto \alpha_{\eta,t}$, 
that is, 
\[
\beta(a)(t):=\begin{cases}
\tilde\alpha_\xi(a)(2t) & 0\leq t\leq1/2 \\
\tilde\alpha_\eta(a)(2t{-}1) & 1/2\leq t\leq1. 
\end{cases}
\]
Therefore 
\begin{align*}
KK(\tilde\alpha_{\xi+\eta})-KK(j)
&=KK(\beta)-KK(j) \\
&=\left(KK(\tilde\alpha_\xi)-KK(j)\right)
+\left(KK(\tilde\alpha_\eta)-KK(j)\right), 
\end{align*}
thereby completing the proof. 
\end{proof}

\begin{definition}\label{kg}
We write the homomorphism constructed above as 
\[
k_\cG:H^1(\cG)\to KK^1(C^*_r(\cG),C^*_r(\cG)), 
\]
and call it the $\T$-action map. 
\end{definition}

This construction will be used to pair 
first cohomology classes with $K$-theory classes.

\subsection{HK}

\begin{definition}\label{HK}
Let $\cG$ be an ample groupoid. 
We say that HK holds for $\cG$ if there exists an isomorphism 
\[
\bigoplus_{j=0}^\infty H_{i+2j}(\cG)\cong K_i(C^*_r(\cG))
\]
for each $i=0,1$. 
\end{definition}

Although HK is known to hold for many important classes of 
ample groupoids, 
it is false in full generality. 
In contrast, the low-degree comparison maps 
\[
\mu_0:H_0(\cG)\to K_0(C^*_r(\cG)),\qquad 
\mu_1:H_1(\cG)\to K_1(C^*_r(\cG))
\]
are canonical, 
and it is natural to ask whether they are injective 
under appropriate hypotheses. 
Some of our results below will be formulated 
in this weaker direction.

\subsection{GL}

Assume that $\cG^{(0)}$ is compact. 
For any $\nu\in M(\cG)$, 
one can associate a tracial state $\tau_\nu$ 
on the $C^*$-algebra $C^*_r(\cG)$ by 
\[
\tau_\nu(f):=\int_{\cG^{(0)}}f(x)\ d\nu(x)
\quad\forall f\in C_c(\cG). 
\]
In general, a tracial state $\tau$ on a $C^*$-algebra $A$ induces 
a homomorphism $\tau_*:\:K_0(A)\to\R$ 
satisfying $\tau_*([p])=\tau(p)$ for every projection $p\in A$. 
By the definition of $\mu_0$, one has 
\[
\tau_{\nu,*}(\mu_0([f]))=\int_{\cG^{(0)}}f(x)\ d\nu(x)
\]
for all $f\in C(\cG^{(0)},\Z)$ and $\nu\in M(\cG)$. 
We let $\Aff(M(\cG))$ denote the space of 
$\R$-valued affine continuous functions on $M(\cG)$. 
Define a homomorphism $D_\cG:K_0(C^*_r(\cG))\to\Aff(M(\cG))$ 
by 
\[
D_\cG(c)(\nu):=\tau_{\nu,*}(c)
\]
for $c\in K_0(C^*_r(\cG))$ and $\nu\in M(\cG)$. 
We call $D_\cG$ the dimension map for $\cG$. 
We emphasize that $D_{\cG}$ is defined 
using only the canonical traces 
arising from invariant probability measures; 
we do not assume that 
every tracial state on $C_r^*(\cG)$ is of this form. 

The gap-labelling in the setting of ample groupoids is formulated 
as follows. 

\begin{definition}\label{GL}
Let $\cG$ be an ample groupoid with $\cG^{(0)}$ compact. 
We say that GL holds for $\cG$ 
if $D_\cG(\mu_0(H_0(\cG)))$ equals $D_\cG(K_0(C^*_r(\cG)))$. 
\end{definition}

We note that, in \cite[Remark 3.8]{Ma16Adv}, 
it was asked whether GL holds 
for any principal, minimal, almost finite groupoid.

\section{Homology comparison maps}

This section establishes the compatibility of 
the homology comparison maps $\mu_0$ and $\mu_1$ (see Section 2.2) 
with the homology long exact sequence and 
the Pimsner--Voiculescu exact sequence in $K$-theory. 

Let $\cH$ be an ample groupoid with compact unit space 
and let $\theta:\cH\to\cH$ be an automorphism. 
Put $\cG:=\cH\rtimes_\theta\Z$. 
Let $\bar\theta\in\Aut(C^*_r(\cH))$ 
denote the automorphism of $C^*_r(\cH)$ induced by $\theta$. 
Thus, $\bar\theta(f):=f\circ\theta^{-1}$ for $f\in C_c(\cH)$. 
The crossed product $C^*$-algebra $C^*_r(\cH)\rtimes_{\bar\theta}\Z$ 
is canonically isomorphic to 
the groupoid $C^*$-algebra $C^*_r(\cH\rtimes_\theta\Z)$, 
and we shall identify these two algebras. 

Consider the following diagram: 
\begin{equation*}
\vcenter{
\xymatrix@M=8pt{
\ar[r] & 
H_1(\cH) \ar[r]^{\id-H_1(\theta)} \ar[d]_{\mu_1} 
\ar@{}[dr]|-{(1)} & 
H_1(\cH) \ar[r] \ar[d]_{\mu_1} \ar@{}[dr]|-{(2)} & 
H_1(\cG) \ar[r]^{\delta} \ar[d]_{\mu_1} \ar@{}[dr]|-{(3)} & 
\phantom{H_0(\cH)} & \\
\ar[r] & 
K_1(C^*_r(\cH)) \ar[r]^{\id-K_1(\bar\theta)} & 
K_1(C^*_r(\cH)) \ar[r] & 
K_1(C^*_r(\cG)) \ar[r]^{\partial} & \phantom{K_0(\cH)} & 
}}
\end{equation*}
\begin{equation}\label{HKdiagram}
\vcenter{
\xymatrix@M=8pt{
\phantom{H} \ar[r]^-{\delta} \ar@{}[dr]|-{(3)} & 
H_0(\cH) \ar[r]^{\id-H_0(\theta)} \ar[d]_{\mu_0} 
\ar@{}[dr]|-{(4)} &
H_0(\cH) \ar[r] \ar[d]_{\mu_0} \ar@{}[dr]|-{(5)} & 
H_0(\cG) \ar[r] \ar[d]_{\mu_0} & 0 \\
\phantom{H} \ar[r]^-{\partial} & 
K_0(C^*_r(\cH)) \ar[r]^{\id-K_0(\bar\theta)} &
K_0(C^*_r(\cH)) \ar[r] & K_0(C^*_r(\cG)) \ar[r], & 
}}
\end{equation}
where the top horizontal sequence is 
a part of the long exact sequence in Lemma \ref{hoLES} 
and the bottom horizontal sequence is 
a part of the Pimsner--Voiculescu exact sequence in $K$-theory. 

\begin{proposition}\label{HKcommute}
Assume further that 
the ample groupoid $\cG$ is either almost finite or purely infinite. 
Then the diagram \eqref{HKdiagram} is commutative. 
\end{proposition}

\begin{proof}
The commutativity of the squares (1) and (4) follows 
from the naturality of $\mu_i$. 
The squares (2) and (5) commute by the definitions of the maps 
$H_i(\cH)\to H_i(\cG)$ and $K_i(C^*_r(\cH))\to K_i(C^*_r(\cG))$. 

It remains to check the square (3). 
By \cite[Theorem 7.5]{Ma12PLMS} and \cite[Theorem 5.2]{Ma15crelle}, 
any element in $H_1(\cG)$ is represented 
by the indicator function of a full bisection $W\subset\cG$. 
Write it as 
\[
W=\bigsqcup_{k\in\Z}W_k\times\{k\}\subset\cH\times\Z=\cG. 
\]
Here each $W_k$ is a compact open bisection of $\cH$ and 
all but finitely many $W_k$ are empty. 
Moreover, 
\[
r(W_k\times\{k\})=r(W_k)\times\{0\},\quad 
s(W_k\times\{k\})=\theta^{-k}(s(W_k))\times\{0\}, 
\]
\[
\bigsqcup_kr(W_k)=\cH^{(0)}=\bigsqcup_k\theta^{-k}(s(W_k)). 
\]
By Lemma \ref{hoLES} and its proof, we have 
\begin{align*}
\delta([1_W])
&=\sum_{k>0}\left(\id+H_0(\theta^{-1})+\dots+H_0(\theta^{1-k})\right)
([1_{s(W_k)}])\\
&\quad -\sum_{k<0}\left(H_0(\theta)+H_0(\theta^2)
+\dots+H_0(\theta^{-k})\right)([1_{s(W_k)}]). 
\end{align*}
This formula is obtained by lifting $1_W$ to 
the skew product used in the proof of Lemma 2.2 and 
computing its boundary there. 
Regard $1_W$ as a unitary in $C^*_r(\cG)$; then 
$\mu_1([1_W])=K_1(1_W)$. 
To compute $\partial(K_1(1_W))$, 
we recall the construction of 
the Pimsner--Voiculescu exact sequence. 
We denote by $u$ the implementing unitary in 
$C^*_r(\cH)\rtimes_{\bar\theta}\Z=C^*_r(\cH\rtimes_\theta\Z)$. 
We use the convention that the implementing unitary $u$ 
satisfies $uau^*=\bar\theta(a)$ for all $a\in C^*_r(\cH)$. 
Let $v$ be the unilateral shift on $\ell^2(\N)$ 
and let 
\[
A:=C^*(C^*_r(\cH)\otimes1,u\otimes v)
\subset(C^*_r(\cH)\rtimes_{\bar\theta}\Z)\otimes C^*(v). 
\]
This yields the following short exact sequence of $C^*$-algebras: 
\[
\xymatrix@M=8pt{
0 \ar[r] & C^*_r(\cH)\otimes\K \ar[r] & 
A \ar[r]^-{\pi} & C^*_r(\cH)\rtimes_{\bar\theta}\Z \ar[r] & 0, 
}
\]
where $\K$ denotes the algebra of compact operators. 
The associated $6$-term exact sequence in $K$-theory gives 
the Pimsner--Voiculescu exact sequence. 
Under the identification 
$C^*_r(\cG)=C^*_r(\cH)\rtimes_{\bar\theta}\Z$, 
the unitary $1_W$ is equal to 
\[
w:=\sum_{k\in\Z}1_{W_k}u^k. 
\]
Set 
\[
x:=\sum_{k\geq0}(1_{W_k}u^k)\otimes v^k
+\sum_{k<0}(1_{W_k}u^k)\otimes(v^*)^{-k}\in A
\]
so that $\pi(x)=w$. 
Put $q_k:=v^{k-1}(1-vv^*)(v^*)^{k-1}$ for $k\in\N$. 
A direct computation gives 
\begin{align*}
xx^*&=\sum_{k\geq0}1_{r(W_k)}\otimes v^kv^{*k}
+\sum_{k<0}1_{r(W_k)}\otimes1\\
&=1-\sum_{k>0}1_{r(W_k)}\otimes(q_1+q_2+\dots+q_k)
\end{align*}
and 
\begin{align*}
x^*x&=\sum_{k\geq0}\bar\theta^{-k}(1_{s(W_k)})\otimes1
+\sum_{k<0}\bar\theta^{-k}(1_{s(W_k)})\otimes v^{-k}(v^*)^{-k}\\
&=1-\sum_{k<0}\bar\theta^{-k}(1_{s(W_k)})\otimes(q_1+q_2+\dots+q_{-k}). 
\end{align*}
For any projection $e\in C^*_r(\cH)$, 
$e\otimes q_i$ is Murray--von Neumann equivalent to 
$\bar\theta^{1-i}(e)\otimes q_1$, 
and hence 
\begin{align*}
\partial(K_1(1_W))
&=\sum_{k>0}\left(\id+K_0(\bar\theta^{-1})+\dots
+K_0(\bar\theta^{1-k})\right)(K_0(1_{r(W_k)}))\\
&\quad -\sum_{k<0}\left(K_0(\bar\theta)+K_0(\bar\theta^2)
+\dots+K_0(\bar\theta^{-k})\right)(K_0(1_{s(W_k)})). 
\end{align*}
Since $W_k$ is a compact open bisection in $\cH$, 
$[1_{r(W_k)}]=[1_{s(W_k)}]$ in $H_0(\cH)$, 
and this equality remains true after applying $H_0(\theta^m)$. 
Thus the preceding formula agrees with $\mu_0(\delta([1_W]))$. 
\end{proof}

The rest of the section is devoted to a technical proposition 
needed later for GL. 
We begin by recalling the de la Harpe--Skandalis determinant 
from \cite{HS84AnnInst}. 
For a unital $C^*$-algebra $A$, 
we denote by $T(A)$ the space of tracial states on $A$. 
The de la Harpe--Skandalis determinant is a homomorphism 
\[
\Delta:U(A)_0\to\Aff(T(A))/D_A(K_0(A)), 
\]
where $U(A)_0$ is the connected component of the identity 
in the unitary group $U(A)$ of $A$ and 
$D_A:K_0(A)\to\Aff(T(A))$ is the homomorphism defined by 
$D_A(c)(\tau)=\tau_*(c)$. 
For $u\in U(A)_0$, the value of $\Delta(u)$ is defined as follows. 
Let $\gamma:[0,1]\to U(A)_0$ be a piecewise smooth path 
such that $\gamma(0)=1$ and $\gamma(1)=u$. 
Then the function $d_\gamma:T(A)\to\R$ defined by 
\[
d_\gamma(\tau)
:=\frac{1}{2\pi i}\int_0^1\tau(\gamma'(t)\gamma(t)^*)\ dt
\]
belongs to $\Aff(T(A))$, 
and $\Delta(u)$ is the class of $d_\gamma$ 
in $\Aff(T(A))/D_A(K_0(A))$. 
It is known that $\Delta$ is a well-defined homomorphism. 

If $\cH$ is a principal ample groupoid with compact unit space 
and $A=C^*_r(\cH)$, then, 
under the identification of $M(\cH)$ with $T(A)$, 
$D_\cH$ agrees with $D_A$. 

\begin{lemma}
Suppose that $\cH$ is an ample groupoid with compact unit space 
which is principal and almost finite. 
Set $A:=C^*_r(\cH)$. 
Let $U\subset\cH$ be a full bisection 
satisfying $[1_U]=0$ in $H_1(\cH)$. 
Then $1_U$, regarded as a unitary in $U(A)$, 
belongs to $U(A)_0$, and 
there exists a piecewise smooth path $\gamma:[0,1]\to U(A)_0$ 
such that $\gamma(0)=1$, $\gamma(1)=1_U$ and 
\[
d_\gamma\in\frac{1}{2}\Ima(D_A\circ\mu_0). 
\]
In particular, $2\Delta(1_U)=0$. 
\end{lemma}

\begin{proof}
By \cite[Theorem 3.3 (4)]{Ma16Adv}, 
the full bisection $U$ can be written as a product of 
finitely many full bisections $U_1,U_2,\dots,U_n$, 
each of which is a transposition. 
Thus, for each $i$, 
there exists a compact open bisection $W_i\subset\cH$ such that 
$r(W_i)\cap s(W_i)=\emptyset$ and 
\[
U_i=W_i\sqcup W_i^{-1}
\sqcup\left(\cH^{(0)}\setminus(r(W_i)\cup s(W_i))\right). 
\]
Let $p_i:=1_{s(W_i)}\in C(\cH^{(0)})\subset A$. 
The unitary $1_{U_i}$ is unitarily equivalent to $1-2p_i$ in $A$. 
Indeed, on the corner corresponding to $r(W_i)\sqcup s(W_i)$, 
the unitary $1_{U_i}$ is the flip matrix, 
which is unitarily equivalent to $\operatorname{diag}(1,-1)$. 
Choose a smooth path $\gamma_i:[0,1]\to U(A)_0$ 
such that $\gamma_i(0)=1$, $\gamma_i(1)=1-2p_i$ and 
$d_{\gamma_i}=D_A(K_0(p_i))/2$. 
Combining these paths gives a piecewise smooth path 
$\gamma:[0,1]\to U(A)_0$ 
such that $\gamma(0)=1$, $\gamma(1)=1_U$ and 
\[
d_\gamma=\sum_{i=1}^n\frac{1}{2}D_A(K_0(p_i))
\in\frac{1}{2}\Ima(D_A\circ\mu_0), 
\]
as desired. 
\end{proof}

Let $\cG=\cH\rtimes_\theta\Z$ be as before. 
Consider the following diagram: 
\begin{equation}\label{GLdiagram}
\vcenter{
\xymatrix@M=8pt{
&&
H_1(\cH) \ar[r]^-{\id-H_1(\theta)} \ar[d]^-{\mu_1} & 
H_1(\cH) \ar[d]^-{\mu_1} \\
K_0(C^*_r(\cH)) \ar[r]^-{\iota} & 
K_0(C^*_r(\cG)) \ar[r]^-{\partial} & 
K_1(C^*_r(\cH)) \ar[r]^-{\id-K_1(\bar\theta)} & 
K_1(C^*_r(\cH)). 
}}
\end{equation}

\begin{proposition}\label{halfGL}
Suppose that $\cH$ is principal and almost finite. 
If $c\in K_0(C^*_r(\cG))$ and $h\in H_1(\cH)$ satisfy 
$\partial(c)=\mu_1(h)$ and $h=H_1(\theta)(h)$, then 
\[
D_\cG(c)\in\frac{1}{2}\Ima(D_\cG\circ\mu_0^\cG)+\Ima(D_\cG\circ\iota). 
\]
\end{proposition}

\begin{proof}
Set $A:=C^*_r(\cH)$. 
By \cite[Theorem 7.5]{Ma12PLMS}, 
we may find a full bisection $W\subset\cH$ such that $[1_W]=h$. 
It follows from $h=H_1(\theta)(h)$ that 
\[
0=[1_W]-H_1(\theta)([1_W])=[1_W]-[1_{\theta(W)}]
=[1_{W\theta(W)^{-1}}]
\]
in $H_1(\cH)$. 
Hence the preceding lemma gives a piecewise smooth path 
$\gamma:[0,1]\to U(A)_0$ 
such that $\gamma(0)=1$, $\gamma(1)=1_W\bar\theta(1_W)^*$ and 
\[
d_\gamma\in\frac{1}{2}\Ima(D_A\circ\mu_0^\cH). 
\]
On the other hand, 
the trace formula for the Pimsner--Voiculescu boundary map 
\cite[Theorem 10.10.4]{Bl_Ktheory} provides an element 
$c_0\in K_0(C_r^*(\cH))$ such that 
\[
\tau_{\nu,*}(c)-d_\gamma\left(\tau_\nu|_A\right)
=\tau_{\nu,*}\left(\iota(c_0)\right)
\]
for every $\nu\in M(\cG)$. 
By the lemma above, the affine function on $M(\cG)$ given by 
$\nu\mapsto d_\gamma\left(\tau_\nu|_A\right)$ belongs to 
\[
\frac{1}{2}\Ima(D_{\cG}\circ\mu_0^{\cG}). 
\]
Indeed, this follows from the compatibility of $\mu_0$ 
with the natural maps $H_0(\cH)\to H_0(\cG)$ and 
$C_r^*(\cH)\to C_r^*(\cG)$. 
Consequently, 
\[
D_{\cG}(c)\in
\frac{1}{2}\Ima(D_{\cG}\circ\mu_0^{\cG})+\Ima(D_{\cG}\circ\iota). 
\]
\end{proof}

\section{Cohomology comparison maps}

Let $\Gamma$ be a countable discrete group. 
Fix $i\in\{0,1\}$ and $\sfb\in K_i(C^*_r(\Gamma))$. 
In this section, 
we would like to introduce the cohomology comparison maps 
\[
\mu^0:H^0(\cG)\to K_i(C^*_r(\cG))
\]
and 
\[
\mu^1:H^1(\cG)\to K_{i+1}(C^*_r(\cG)) 
\]
associated with $\sfb$, 
where $\cG:=X\rtimes\Gamma$ is a transformation groupoid 
on a Cantor set $X$. 

Consider an action $\Gamma\curvearrowright X$ 
on a Cantor set $X$ 
and let $\cG:=X\rtimes\Gamma$ be the transformation groupoid. 
We first define $\mu^0$. 
The group $H^0(\cG)$ identifies with $C(X,\Z)^\Gamma$, 
the group of $\Gamma$-invariant functions, and 
there exists a canonical map $C(X,\Z)^\Gamma\to K_0(C(X)^\Gamma)$ 
sending an integer-valued invariant function 
to the corresponding $K_0$-class. 
So $g\mapsto g\otimes\sfb$ defines a homomorphism 
from $H^0(\cG)\cong C(X,\Z)^\Gamma$ 
to $K_0(C(X)^\Gamma)\otimes K_i(C^*_r(\Gamma))$. 
Since $C(X)^\Gamma\otimes C^*_r(\Gamma)$ can be identified 
with a subalgebra of $C^*_r(\cG)$ 
via $f\otimes\lambda_\gamma\mapsto f u_\gamma$, 
we have a homomorphism 
from $K_0(C(X)^\Gamma)\otimes K_i(C^*_r(\Gamma))$
to $K_i(C^*_r(\cG))$. 
Let $\mu^0:H^0(\cG)\to K_i(C^*_r(\cG))$ 
be the composition of these homomorphisms. 

Now, let us define $\mu^1:H^1(\cG)\to K_{i+1}(C^*_r(\cG))$. 
Recall the $\T$-action map 
$k_\cG:H^1(\cG)\to KK^1(C^*_r(\cG),C^*_r(\cG))$ 
introduced in Definition \ref{kg}. 
For $\xi\in\Hom(\cG,\Z)$, we let 
\[
\mu^1([\xi]):=k_\cG([\xi])(\sfb)\in K_{i+1}(C^*_r(\cG)), 
\]
where $\sfb\in K_i(C^*_r(\Gamma))$ is regarded as 
an element of $K_i(C^*_r(\cG))$ 
via the embedding $C^*_r(\Gamma)\to C^*_r(\cG)$. 
We call $\mu^0$ and $\mu^1$ the cohomology comparison maps 
associated with $\sfb\in K_i(C^*_r(\Gamma))$.

Let $\Lambda$ be a countable discrete group and 
let $\Gamma:=\Lambda\rtimes_\theta\Z$ be a semidirect product 
by an automorphism $\theta\in\Aut(\Lambda)$. 
We fix $i\in\{0,1\}$, 
$\sfb_\Lambda\in K_i(C^*_r(\Lambda))$ and 
$\sfb_\Gamma\in K_{i+1}(C^*_r(\Gamma))$. 

Suppose that $\Gamma$ acts on a Cantor set $X$. 
Let $\cG:=X\rtimes\Gamma$ be the transformation groupoid of 
this action and 
let $\cH:=X\rtimes\Lambda\subset\cG$ be the subgroupoid 
associated with the restriction $\Lambda\curvearrowright X$. 
We may think of $\cG$ as a semidirect product of $\cH$ 
by an automorphism $\cH\to\cH$, 
which is an `extension' of $\theta$. 
With a slight abuse of notation, 
we write this automorphism as $\theta:\cH\to\cH$ as well. 
Let $\mu^j$ be the associated cohomology comparison maps. 
To lighten notation, 
we shall suppress the superscripts indicating the ambient groupoid. 
Thus the cohomology comparison maps 
associated with $\sfb_\Gamma$ for $\cG$ and 
those associated with $\sfb_\Lambda$ for $\cH$ 
will both be denoted by $\mu^j$. 
Strictly speaking, 
they should be written as $\mu^j_\cG$ and $\mu^j_\cH$, respectively. 
We also point out that their degrees are different, 
reflecting the fact that $\sfb_\Gamma\in K_{i+1}(C^*_r(\Gamma))$ 
whereas $\sfb_\Lambda\in K_i(C^*_r(\Lambda))$. 
With this convention, consider the following diagram: 
\begin{equation*}
\vcenter{
\xymatrix@M=8pt{
0 \ar[r] & 
H^0(\cG) \ar[r] \ar[d]_{\mu^0} \ar@{}[dr]|-{(1)} & 
H^0(\cH) \ar[r]^-{\id-H^0(\theta)} \ar[d]_{\mu^0} 
\ar@{}[dr]|-{(2)} & 
H^0(\cH) \ar[r]^-{\delta} \ar[d]_{\mu^0} \ar@{}[dr]|-{(3)} & 
\phantom{H^1(\cG)} \\
\ar[r] & 
K_{i+1}(C^*_r(\cG)) \ar[r]^-{\partial_{i+1}}  & 
K_{i}(C^*_r(\cH)) \ar[r]^-{\id-K_{i}(\bar\theta)} & 
K_{i}(C^*_r(\cH)) \ar[r]^-{\iota} & \phantom{H^1(\cG)}
}}
\end{equation*}
\begin{equation}\label{coHKdiagram}
\vcenter{
\xymatrix@M=8pt{
\phantom{H} \ar[r]^-{\delta} \ar@{}[dr]|-{(3)} & 
H^1(\cG) \ar[r] \ar[d]_{\mu^1} \ar@{}[dr]|-{(4)} &
H^1(\cH) \ar[r]^-{\id-H^1(\theta)} \ar[d]_{\mu^1} 
\ar@{}[dr]|-{(5)} & 
H^1(\cH) \ar[d]_{\mu^1} \ar[r] & \\
\phantom{H} \ar[r]^-{\iota} & 
K_{i}(C^*_r(\cG)) \ar[r]^-{\partial_i} & 
K_{i+1}(C^*_r(\cH)) \ar[r]^-{\id-K_{i+1}(\bar\theta)} & 
K_{i+1}(C^*_r(\cH)) \ar[r] & 
}}
\end{equation}
where the top horizontal sequence is 
a part of the long exact sequence in Lemma \ref{coLES}, 
the bottom horizontal sequence is 
a part of the Pimsner--Voiculescu exact sequence in $K$-theory, and 
the automorphism of $C^*_r(\cH)$ induced by $\theta\in\Aut(\cH)$ 
is denoted by $\bar\theta$. 

\begin{proposition}\label{coHKcommute}
Suppose that the connecting map 
$K_{i+1}(C^*_r(\Gamma))\to K_i(C^*_r(\Lambda))$ 
in the Pimsner--Voiculescu exact sequence 
for $C^*_r(\Gamma)=C^*_r(\Lambda)\rtimes_{\bar\theta}\Z$ 
sends $\sfb_\Gamma$ to $\sfb_\Lambda$. 
Then the diagram \eqref{coHKdiagram} is commutative. 
\end{proposition}

\begin{proof}
The commutativity of the square (1) follows from the diagram: 
\begin{equation*}
\vcenter{
\xymatrix@M=8pt{
H^0(\cG) \ar[r] \ar[d] \ar@{}[dr]|(.5){(1')} & 
H^0(\cH) \ar[d] & \\
K_0(C(X)^\Gamma)\otimes K_{i+1}(C^*_r(\Gamma)) \ar[r] \ar[d] & 
K_0(C(X)^\Lambda)\otimes K_{i}(C^*_r(\Lambda)) \ar[d] \\
K_{i+1}(C^*_r(\cG)) \ar[r]^-{\partial_{i+1}}  & 
K_{i}(C^*_r(\cH)). 
}}
\end{equation*}
The square $(1')$ commutes 
because of our assumption on $\sfb_\Gamma$ and $\sfb_\Lambda$. 
Moreover, by exactness, this implies that $\sfb_\Lambda$, 
when viewed as an element in $K_i(C^*_r(\cH))$, 
is in the kernel of $\id-K_i(\bar\theta)$. 
Thus, one has $K_i(\bar\theta)(\sfb_\Lambda)=\sfb_\Lambda$. 
It follows that the square (2) is also commutative. 

We would like to show that the square (3) is commutative. 
First, let us consider the case $i=0$, that is, 
$\sfb_\Lambda$ belongs to $K_0(C^*_r(\Lambda))$. 
There exist natural isomorphisms 
$K_0(C^*_r(\Lambda))\cong K_0(C^*_r(\Lambda)\otimes\cO_\infty)$ and 
\[
K_0(C^*_r(\Gamma))\cong K_0(C^*_r(\Gamma)\otimes\cO_\infty)
=K_0((C^*_r(\Lambda)\rtimes_{\bar\theta}\Z)\otimes\cO_\infty)
=K_0((C^*_r(\Lambda)\otimes\cO_\infty)
\rtimes_{\bar\theta\otimes\id}\Z), 
\]
and so we may work in the $C^*$-algebras tensored with $\cO_\infty$. 
By \cite[Proposition 4.1.4]{Ro_text}, 
we can find a properly infinite, full projection 
$p\in C^*_r(\Lambda)\otimes\cO_\infty$ 
whose $K_0$-class equals $\sfb_\Lambda$. 
Since $K_0(\bar\theta)$ fixes $\sfb_\Lambda$, 
by \cite[Proposition 4.1.4]{Ro_text}, 
there exists a unitary $v\in C^*_r(\Lambda)\otimes\cO_\infty$ 
such that $v\bar\theta(p)v^*=p$ and the $K_1$-class of the unitary 
\[
w:=vup+(1-p)\in C^*_r(\Gamma)\otimes\cO_\infty
\]
equals $\sfb_\Gamma$, 
where 
\[
u\in C^*_r(\Gamma)\otimes\cO_\infty
=(C^*_r(\Lambda)\otimes\cO_\infty)
\rtimes_{\bar\theta\otimes\id}\Z
\]
denotes the implementing unitary. 
Indeed, first choose $v_0$ with $v_0\bar\theta(p)v_0^*=p$. 
Then the class of $v_0up+(1-p)$ has boundary 
$K_0(p)=\sfb_\Lambda$, 
and hence differs from $\sfb_\Gamma$ by an element 
coming from $K_1(C^*_r(\Lambda)\otimes\cO_\infty)$. 
Since $p$ is full, the inclusion $pAp\to A$, 
where $A=C^*_r(\Lambda)\otimes\cO_\infty$, 
induces an isomorphism on $K_1$. 
Thus we may choose a unitary in the corner $pAp$ 
representing the required correction, 
and after replacing $v_0$ by the corrected unitary 
we obtain $v$ with $K_1(w)=\sfb_\Gamma$. 

Since $C(X,\Z)^\Lambda$ is generated by 
the indicator functions of $\Lambda$-invariant clopen subsets, 
it suffices to consider $g=1_A$. 
Then $\iota(\mu^0(g))\in K_0(C^*_r(\cG))$ is given 
by the projection $gp\in C^*_r(\cG)\otimes\cO_\infty$ 
(precisely speaking, $g$ should be $g\otimes1_{\cO_\infty}$, 
but we adopt simplified notation). 
On the other hand, 
one can find a homomorphism $\xi:\cG=X\times\Lambda\times\Z\to\Z$ 
satisfying 
\[
\xi(x,l,0)=0\quad\text{and}\quad 
\xi(x,l,1)=g(x)\quad\forall (x,l)\in\cH=X\times\Lambda, 
\]
and $\delta(g)$ is equal to 
the equivalence class of $\xi$ in $H^1(\cG)$. 
Let $\alpha_\xi:\T\curvearrowright C^*_r(\cG)$ be as in Section 2.2. 
By the definition of $\xi$, 
$\alpha_{\xi,t}(a)=a$ for all $a\in C^*_r(\cH)$ and 
\[
\alpha_{\xi,t}(u)=e^{2\pi itg}u. 
\]
Since $g$ is $\Lambda$-invariant, it commutes with $C^*_r(\Lambda)$. 
Hence
\[
\alpha_{\xi,t}(w)w^*=e^{2\pi itg}p+(1-p)
=\exp(2\pi it\, gp).
\]
Consequently, 
the loop $t\mapsto \alpha_{\xi,t}(w)w^*$ represents $K_0(gp)$, 
which means 
\[
\mu^1(\delta(g))=\mu^1([\xi])
=k_\cG([\xi])(\sfb_\Gamma)
=K_0(gp)=\iota(\mu^0(g)). 
\]
Thus, we have completed the proof of the case $i=0$. 
When $i=1$, the element $\sfb_\Lambda$ is in $K_1(C^*_r(\Lambda))$. 
Let $\mathcal{P}$ be the unital UCT Kirchberg algebra 
such that $K_0(\mathcal{P})=0$ and $K_1(\mathcal{P})=\Z$. 
Tensoring $\mathcal{P}$ instead of $\cO_\infty$ 
shifts the degree by one. 
Under the resulting identifications, 
the preceding argument applies verbatim. 

Now we turn to the square (4). 
Pick a homomorphism $\xi:\cG\to\Z$. 
Let $\eta:\cH\to\Z$ be the restriction of $\xi$ to $\cH$. 
Then, 
the homomorphism $H^1(\cG)\to H^1(\cH)$ sends $\xi$ to $\eta$. 
As described in Section 2.2, we can define the two homomorphisms 
\[
\tilde\alpha_\xi:C^*_r(\cG)\to C(\T)\otimes C^*_r(\cG)
\]
and 
\[
\tilde\alpha_\eta:C^*_r(\cH)\to C(\T)\otimes C^*_r(\cH). 
\]
Consider the following commutative diagram: 
\begin{equation*}
\vcenter{
\xymatrix@M=8pt{
\ar[r] & 
K_{i+1}(C^*_r(\cG)) \ar[r]^-{\partial_{i+1}} 
\ar[d]_-{K_{i+1}(\tilde\alpha_\xi)-K_{i+1}(j)} & 
K_{i}(C^*_r(\cH)) \ar[r] 
\ar[d]^-{K_{i}(\tilde\alpha_\eta)-K_{i}(j)} & & \\
\ar[r] & 
K_{i+1}(C_0(\R)\otimes C^*_r(\cG)) \ar[r]^-{\partial_{i}} & 
K_{i}(C_0(\R)\otimes C^*_r(\cH)) \ar[r], & & 
}}
\end{equation*}
where each horizontal sequence is 
a part of the Pimsner--Voiculescu exact sequence. 
The diagram commutes by the naturality of the PV exact sequence. 
We compute 
\begin{align*}
\partial_i(\mu^1([\xi]))
&=\partial_i\left(k_\cG([\xi])(\sfb_\Gamma)\right) \\
&=\partial_i\left(
(K_{i+1}(\tilde\alpha_\xi)-K_{i+1}(j))(\sfb_\Gamma)\right) \\
&=(K_{i}(\tilde\alpha_\eta)-K_{i}(j))
\left(\partial_{i+1}(\sfb_\Gamma)\right) \\
&=(K_{i}(\tilde\alpha_\eta)-K_{i}(j))(\sfb_\Lambda) \\
&=k_\cH([\eta])(\sfb_\Lambda)=\mu^1([\eta]), 
\end{align*}
and so the square (4) is commutative. 

Finally, in order to check (5), take $\eta\in\Hom(\cH,\Z)$. 
For any $a\in C_c(\cH)$, $t\in\T$ and $g\in\cH$, it holds that 
\begin{align*}
\alpha_{\eta\circ\theta^{-1},t}(a)(g)
&=e^{2\pi it\eta(\theta^{-1}(g))}a(g) \\
&=e^{2\pi it\eta(\theta^{-1}(g))}
\bar\theta^{-1}(a)(\theta^{-1}(g)) \\
&=\alpha_{\eta,t}(\bar\theta^{-1}(a))(\theta^{-1}(g))
=\left(\bar\theta\circ\alpha_{\eta,t}\circ\bar\theta^{-1}\right)
(a)(g). 
\end{align*}
Hence 
\begin{align*}
k_\cH(H^1(\theta)([\eta]))
&=KK(\tilde\alpha_{\eta\circ\theta^{-1}})-KK(j) \\
&=KK\left((\id\otimes\bar\theta)\circ\tilde\alpha_{\eta}
\circ\bar\theta^{-1}\right)-KK(j) \\
&=KK(\id\otimes\bar\theta)\cdot(KK(\tilde\alpha_\eta)-KK(j))
\cdot KK(\bar\theta^{-1}) \\
&=KK(\id\otimes\bar\theta)\cdot k_\cH([\eta])
\cdot KK(\bar\theta^{-1}). 
\end{align*}
It follows from $K_i(\bar\theta)(\sfb_\Lambda)=\sfb_\Lambda$ that 
\begin{align*}
\mu^1(H^1(\theta)([\eta]))
&=k_\cH(H^1(\theta)([\eta]))(\sfb_\Lambda) \\
&=\left(KK(\id\otimes\bar\theta)\cdot k_\cH([\eta])
\cdot KK(\bar\theta^{-1})\right)(\sfb_\Lambda) \\
&=\left(KK(\id\otimes\bar\theta)\cdot k_\cH([\eta])\right)
(\sfb_\Lambda) \\
&=KK(\id\otimes\bar\theta)(\mu^1([\eta]))
=K_{i+1}(\bar\theta)(\mu^1([\eta])). 
\end{align*}
Therefore, the square (5) is commutative. 
\end{proof}

\section{Actions of poly-$\Z$ groups}

In this section, 
we discuss the HK (Definition \ref{HK}) and GL (Definition \ref{GL}) 
for free actions of poly-$\Z$ groups of Hirsch length 
at most four, and also for a special class of 
poly-$\Z$ groups of Hirsch length five. 
It is known that 
the transformation groupoids of such actions are almost finite 
(\cite{KN26Compos,Na24JFA}).

\subsection{Poly-$\Z$ groups of Hirsch length less than three}

First, 
we consider the poly-$\Z$ groups of Hirsch length less than $3$, 
i.e. $\Z$, $\Z^2$ and the Klein bottle group. 
Let $\sfb_{\Z^2}\in K_0(C^*_r(\Z^2))\cong\Z^2$ be $K_0(e)-K_0(1)$, 
where $e\in M_2(C^*_r(\Z^2))=M_2(C(\T^2))$ is the Bott projection. 
We choose a generator $\sfb_\Z\in K_1(C^*_r(\Z))\cong\Z$ 
so that, for the decomposition $\Z^2=\Z\rtimes\Z$ 
with trivial action, 
the Pimsner--Voiculescu boundary map sends 
$\sfb_{\Z^2}$ to $\sfb_{\Z}$.

We begin with $\Z$. 
The following is well-known and easy to verify. 

\begin{theorem}\label{Z}
Let $\cG$ be the transformation groupoid 
arising from a free action of $\Z$ on a Cantor set $X$. 
\begin{enumerate}
\item The homology comparison maps 
$\mu_0:H_0(\cG)\to K_0(C^*_r(\cG))$ 
and $\mu_1:H_1(\cG)\to K_1(C^*_r(\cG))$ are isomorphisms. 
\item The cohomology comparison maps 
$\mu^0:H^0(\cG)\to K_1(C^*_r(\cG))$ 
and $\mu^1:H^1(\cG)\to K_0(C^*_r(\cG))$ 
associated with $\sfb_\Z$ are isomorphisms. 
\end{enumerate}
In particular, HK and GL hold true for $\cG$. 
\end{theorem}

\begin{proof}
This follows immediately from the Pimsner--Voiculescu exact sequence 
and the standard computation of the homology and cohomology of 
$\Z$ with coefficients in $C(X,\Z)$. 
\end{proof}

Next let us look at $\Z^2$. 

\begin{theorem}\label{Z^2}
Let $\cG$ be the transformation groupoid 
arising from a free action of $\Z^2$ on a Cantor set $X$. 
Consider the cohomology comparison maps 
associated with $\sfb_{\Z^2}$. 
Then, 
\[
\mu_0\oplus\mu^0:H_0(\cG)\oplus H^0(\cG)\to K_0(C^*_r(\cG)), 
\]
\[
\mu_1:H_1(\cG)\to K_1(C^*_r(\cG))
\]
and 
\[
\mu^1:H^1(\cG)\to K_1(C^*_r(\cG))
\]
are isomorphisms. 
In particular, HK and GL hold true for $\cG$. 
\end{theorem}

\begin{proof}
We regard $\Gamma:=\Z^2$ 
as a semidirect product of $\Lambda:=\Z$ by the trivial action. 
Let $\cH:=X\rtimes\Lambda$. 
By the theorem above, for $i=0,1$, 
$\mu_i:H_i(\cH)\to K_i(C^*_r(\cH))$ is an isomorphism. 
Proposition \ref{HKcommute} then gives that 
$\mu_1:H_1(\cG)\to K_1(C^*_r(\cG))$ is also an isomorphism. 
In the same way, by the theorem above, for $i=0,1$, 
$\mu^i:H^i(\cH)\to K_{i+1}(C^*_r(\cH))$ is an isomorphism. 
Proposition \ref{coHKcommute} then implies that 
$\mu^1:H^1(\cG)\to K_1(C^*_r(\cG))$ is an isomorphism as well. 

To analyze 
$\mu_0\oplus\mu^0:H_0(\cG)\oplus H^0(\cG)\to K_0(C^*_r(\cG))$, 
we combine 
the two diagrams \eqref{HKdiagram} and \eqref{coHKdiagram}. 
\begin{equation*}
\vcenter{
\xymatrix@M=8pt{
H_0(\cH) \ar[r]^{\id-H_0(\theta)} \ar[d]_{\mu_0}^-{\cong} & 
H_0(\cH) \ar[r] \ar[d]_{\mu_0}^-{\cong} & 
H_0(\cG)\oplus H^0(\cG) \ar[r] \ar[d]_-{\mu_0\oplus\mu^0} & 
H^0(\cH) \ar[r]^-{\id-H^0(\theta)} \ar[d]_-{\mu^0}^{\cong} & 
H^0(\cH) \ar[d]_-{\mu^0}^{\cong} \\
K_0(C^*_r(\cH)) \ar[r]^{\id-K_0(\bar\theta)} & 
K_0(C^*_r(\cH)) \ar[r]^-{\iota} & 
K_0(C^*_r(\cG)) \ar[r]^-{\partial} & 
K_1(C^*_r(\cH)) \ar[r]^{\id-K_1(\bar\theta)} & 
K_1(C^*_r(\cH))
}}
\end{equation*}
The five lemma implies that 
$\mu_0\oplus\mu^0$ is an isomorphism. 

By Poincar\'e duality for $\Z^2$ with coefficients in $C(X,\Z)$, 
$H^0(\cG)$ is isomorphic to $H_2(\cG)$, 
and so HK holds for $\cG$. 
For $f\in H^0(\cG)=C(X,\Z)^{\Gamma}$, 
the class $\mu^0(f)$ is obtained from $f\otimes\sfb_{\Z^2}$. 
Since the Bott class $\sfb_{\Z^2}$ has zero trace 
under the canonical trace on $C^*_r(\Z^2)$, 
$D_\cG(\mu^0(f))=0$. 
Hence the image of $\mu^0$ is contained in the kernel of $D_\cG$, 
and so GL holds for $\cG$. 
\end{proof}

Finally, we consider actions of the Klein bottle group $\Gamma$, 
which is the semidirect product of $\Lambda:=\Z$ 
by the nontrivial automorphism $\theta$. 
For the group $\Gamma$, 
we cannot use the cohomology comparison maps 
in the inductive manner above, 
because $\theta$ does not fix $\sfb_\Z$, 
i.e., $K_1(\bar\theta)(\sfb_\Z)=-\sfb_\Z$. 
Nevertheless, 
the $K$-groups of the groupoid $C^*$-algebra are 
calculated as follows. 

\begin{theorem}\label{Klein}
Let $\cG$ be the transformation groupoid 
arising from a free action of the Klein bottle group $\Gamma$ 
on a Cantor set $X$. 
\begin{enumerate}
\item The following sequence is split exact: 
\[
\xymatrix@M=8pt{
0 \ar[r] & 
H_0(\cG) \ar[r]^-{\mu_0} & 
K_0(C^*_r(\cG)) \ar[r] & 
H_2(\cG) \ar[r] & 0. 
}
\]
Thus, $K_0(C^*_r(\cG))$ is isomorphic to $H_0(\cG)\oplus H_2(\cG)$. 
In particular, $\mu_0:H_0(\cG)\to K_0(C^*_r(\cG))$ is injective. 
\item $\mu_1:H_1(\cG)\to K_1(C^*_r(\cG))$ is an isomorphism. 
\end{enumerate}
In particular, HK and GL hold true for $\cG$. 
\end{theorem}

\begin{proof}
Let $\cH=X\rtimes\Lambda$. 
The same argument as in Theorem \ref{Z^2} shows that 
$\mu_1:H_1(\cG)\to K_1(C^*_r(\cG))$ is an isomorphism. 
The proof of Proposition \ref{coHKcommute} implies that 
the following diagram commutes: 
\begin{equation*}
\vcenter{
\xymatrix@M=8pt{
H^0(\cH) \ar[r]^-{\id+H^0(\theta)} \ar[d]_{\mu^0} & 
H^0(\cH) \ar[d]_{\mu^0} \\
K_{1}(C^*_r(\cH)) \ar[r]^-{\id-K_{1}(\bar\theta)} & 
K_{1}(C^*_r(\cH)). \\
}}
\end{equation*}
Hence we obtain the exact sequence 
\[
\xymatrix@M=8pt{
0 \ar[r] & 
H_0(\cG) \ar[r]^-{\mu_0} & 
K_0(C^*_r(\cG)) \ar[r] & 
\Ker(\id+H^0(\theta)) \ar[r] & 0. 
}
\]
Since Poincar\'e duality $H^0(\cH)\cong H_1(\cH)$ intertwines 
$H^0(\theta)$ and $-H_1(\theta)$, 
$\Ker(\id+H^0(\theta))$ is isomorphic to $H_2(\cG)$. 
Since 
\[
\Ker(\id+H^0(\theta))
=\{f\in C(X,\Z)^\Lambda\mid f\circ\theta=-f\}
\]
is free abelian, we obtain (1). 

By (1) and (2), HK holds for $\cG$. 
It remains for us to verify GL. 
Assume that $h\in H^0(\cH)$ belongs to $\Ker(\id+H^0(\theta))$. 
The element $h$ may be regarded as 
a $\Lambda$-invariant function on $X$ 
satisfying $h\circ\theta=-h$. 
There exist $\Lambda$-invariant clopen subsets $P_1,P_2,\dots,P_m$ 
and $Q_1,Q_2,\dots,Q_m$ of $X$ such that 
$P_j\cap Q_j=\emptyset$, $\theta(P_j)=Q_j$, $\theta(Q_j)=P_j$ 
and 
\[
h=\sum_{j=1}^m1_{P_j}-\sum_{j=1}^m1_{Q_j}. 
\]
Let $z\in C^*_r(\Lambda)=C^*_r(\Z)$ be the generating unitary 
representing $\sfb_\Lambda\in K_1(C^*_r(\Lambda))$. 
Then $\bar\theta(z)=z^*$. 
For each $j$, we put 
\[
v_j:=1_{P_j}\otimes z+1_{Q_j}\otimes z^*+(1_X-1_{P_j}-1_{Q_j})
\in C(X)^\Lambda\otimes C^*_r(\Lambda)\subset C^*_r(\cH). 
\]
Then $\mu^0(h)=\sum_jK_1(v_j)$, and 
\[
v_j\bar\theta(v_j^*)=1_{P_j}\otimes1
+1_{Q_j}\otimes1+(1_X-1_{P_j}-1_{Q_j})=1. 
\]
By the trace formula for the Pimsner--Voiculescu boundary map, 
as used in Proposition \ref{halfGL}, 
the only possible obstruction to 
lifting the trace value from $C^*_r(\cH)$ is 
the de la Harpe--Skandalis determinant of $v_j\bar\theta(v_j^*)$. 
Since this unitary is equal to $1$, 
the determinant term vanishes. 
Hence $\Ima D_\cG$ is contained in $\Ima(D_\cG\circ\iota)$. 
\end{proof}

\begin{remark}\label{KleinRemark}
In the theorem above, 
we may describe the splitting $H_2(\cG)\to K_0(C^*_r(\cG))$ 
in the following way. 
As mentioned in the proof, 
\[
H_2(\cG)\cong\Ker(\id+H^0(\theta))
=\{f\in C(X,\Z)^\Lambda\mid f\circ\theta=-f\}
\]
is generated by functions of the form $h=1_{P}-1_{Q}$ 
such that $\theta(P)=Q$ and $\theta(Q)=P$. 
Define a unitary 
$v\in C(X)^\Lambda\otimes C^*_r(\Lambda)\subset C^*_r(\cH)$ 
by 
\[
v:=1_{P}\otimes z+1_{Q}\otimes z^*+(1_X-1_{P}-1_{Q}). 
\]
Let $u\in C^*_r(\Gamma)\subset C^*_r(\cG)$ be the unitary 
implementing the automorphism $\bar\theta$. 
Then $v$ commutes with $u$ because $\bar\theta(v)=v$, 
and the Bott element $\Bott(v,u)\in K_0(C^*_r(\cG))$ 
is a lift of the function $1_P-1_Q\in\Ker(\id+H^0(\theta))$. 
Thus, 
with our convention for the Pimsner--Voiculescu boundary map,
$\partial(\Bott(v,u))=K_1(v)$, 
which corresponds to $1_P-1_Q$ under $\mu^0$. 
\end{remark}

\begin{remark}
The $K$-theoretic part of Theorem \ref{Klein} should be 
compared with earlier low-dimensional results on HK. 
Under the hypotheses of the spectral sequence constructed 
in \cite{PY22ETDS}, 
\cite[Remark 4.5]{PY22ETDS} shows that, 
if $H_k(\cG)=0$ for all $k\geq3$, then 
there is a short exact sequence 
\[
\xymatrix@M=8pt{
0 \ar[r] & 
H_0(\cG) \ar[r]^-{\mu_0} & 
K_0(C^*_r(\cG)) \ar[r] & 
H_2(\cG) \ar[r] & 0
}
\]
and $\mu_1:H_1(\cG)\to K_1(C^*_r(\cG))$ is an isomorphism. 
In particular, 
if $H_2(\cG)$ is free, then the exact sequence splits. 
In a related direction, 
\cite[Theorem 3.36]{BDGW23PLMS} proves that, 
for a principal ample groupoid $\cG$ 
with dynamic asymptotic dimension at most $d$, 
one has $H_n(\cG)=0$ for $n>d$, and 
$H_d(\cG)$ is torsion-free. 
Consequently, 
when the dynamic asymptotic dimension is at most $2$ and 
$H_2(\cG)$ is finitely generated, 
the same $K$-theoretic conclusion, and hence HK, follows; 
see \cite[Theorem 4.19]{BDGW23PLMS}. 
Relevant computations and related results can also be found 
in \cite[Section 6]{FKPS19Munster}. 
\end{remark}

\subsection{Poly-$\Z$ groups of Hirsch length three}

In this subsection, 
we treat actions of poly-$\Z$ groups of Hirsch length $3$. 
We divide these groups into the following three classes. 

First, let us consider a semidirect product $\Gamma$ of 
$\Lambda\cong\Z^2$ by an automorphism $\theta\in\Aut(\Lambda)$. 
The automorphism $\theta$ is represented 
by a matrix $A\in GL(2,\Z)$. 
We distinguish two cases: 
if $\det A=1$ then $K_0(\bar\theta)(\sfb_{\Z^2})=\sfb_{\Z^2}$; 
if $\det A=-1$ then $K_0(\bar\theta)(\sfb_{\Z^2})=-\sfb_{\Z^2}$. 
The last case is when $\Lambda$ is the Klein bottle group. 
We refer to these three classes as (3a), (3b) and (3c), respectively. 

Let us first focus on $\Gamma$ belonging to class (3a). 
Choose $\sfb_\Gamma\in K_1(C^*_r(\Gamma))$ so that 
the connecting map $K_1(C^*_r(\Gamma))\to K_0(C^*_r(\Lambda))$ 
sends $\sfb_\Gamma$ to $\sfb_{\Z^2}$. 

\begin{theorem}\label{Hirsch3a}
Let $\Gamma=\Lambda\rtimes_\theta\Z$ be as above and 
let $\cG$ be the transformation groupoid 
arising from a free action of $\Gamma$ on a Cantor set $X$. 
Then, 
$\mu_0\oplus\mu^1:H_0(\cG)\oplus H^1(\cG)\to K_0(C^*_r(\cG))$ 
and $\mu_1\oplus\mu^0:H_1(\cG)\oplus H^0(\cG)\to K_1(C^*_r(\cG))$ 
are isomorphisms. 
Moreover, HK holds true for $\cG$. 
\end{theorem}

\begin{proof}
Set $\cH:=X\rtimes\Lambda$. 
Let $\Phi_0:=(\id-H_0(\theta))\oplus(\id-H^0(\theta))$. 
The two diagrams \eqref{HKdiagram} and \eqref{coHKdiagram} 
combine to give the following: 
\begin{equation*}
\vcenter{
\xymatrix@M=8pt{
H_0(\cH)\oplus H^0(\cH) \ar[r]^{\Phi_0} 
\ar[d]_{\mu_0\oplus\mu^0}^-{\cong} & 
H_0(\cH)\oplus H^0(\cH) \ar[r] \ar[d]_{\mu_0\oplus\mu^0}^-{\cong} & 
H_0(\cG)\oplus H^1(\cG) \ar[r] \ar[d]_-{\mu_0\oplus\mu^1} & 
\phantom{H} \\
K_0(C^*_r(\cH)) \ar[r]^{\id-K_0(\bar\theta)} & 
K_0(C^*_r(\cH)) \ar[r]^-{\iota} & 
K_0(C^*_r(\cG)) \ar[r]^-{\partial} & 
\phantom{H}
}}
\end{equation*}
\begin{equation*}
\vcenter{
\xymatrix@M=8pt{
\phantom{H} \ar[r] & 
H^1(\cH) \ar[r]^-{\id-H^1(\theta)} \ar[d]_-{\mu^1}^{\cong} & 
H^1(\cH) \ar[d]_-{\mu^1}^{\cong} \\
\phantom{H} \ar[r]^-{\partial} & 
K_1(C^*_r(\cH)) \ar[r]^{\id-K_1(\bar\theta)} & 
K_1(C^*_r(\cH)). 
}}
\end{equation*}
By Theorem \ref{Z^2}, 
$\mu_0\oplus\mu^0:H_0(\cH)\oplus H^0(\cH)\to K_0(C^*_r(\cH))$ and 
$\mu^1:H^1(\cH)\to K_1(C^*_r(\cH))$ are isomorphisms. 
Therefore, the middle vertical map $\mu_0\oplus\mu^1$ 
is an isomorphism. 

Similarly, 
the two diagrams \eqref{HKdiagram} and \eqref{coHKdiagram} induce 
the following commutative diagram: 
\begin{equation*}
\vcenter{
\xymatrix@M=8pt{
H_1(\cH) \ar[r]^{\id-H_1(\theta)} 
\ar[d]_{\mu_1}^-{\cong} & 
H_1(\cH) \ar[r] \ar[d]_{\mu_1}^-{\cong} & 
H_1(\cG)\oplus H^0(\cG) \ar[r] \ar[d]_-{\mu_1\oplus\mu^0} & 
\phantom{H} \\
K_1(C^*_r(\cH)) \ar[r]^{\id-K_1(\bar\theta)} & 
K_1(C^*_r(\cH)) \ar[r]^-{\iota} & 
K_1(C^*_r(\cG)) \ar[r]^-{\partial} & 
\phantom{H}
}}
\end{equation*}
\begin{equation*}
\vcenter{
\xymatrix@M=8pt{
\phantom{H} \ar[r] & 
H_0(\cH)\oplus H^0(\cH) \ar[r]^-{\Phi_0} 
\ar[d]_-{\mu_0\oplus\mu^0}^{\cong} & 
H_0(\cH)\oplus H^0(\cH) \ar[d]_-{\mu_0\oplus\mu^0}^{\cong} \\
\phantom{H} \ar[r]^-{\partial} & 
K_0(C^*_r(\cH)) \ar[r]^{\id-K_0(\bar\theta)} & 
K_0(C^*_r(\cH)). 
}}
\end{equation*}
By Theorem \ref{Z^2}, 
$\mu_1:H_1(\cH)\to K_1(C^*_r(\cH))$ and 
$\mu_0\oplus\mu^0:H_0(\cH)\oplus H^0(\cH)\to K_0(C^*_r(\cH))$ 
are isomorphisms. 
Therefore, the middle vertical map $\mu_1\oplus\mu^0$ 
is an isomorphism. 

By Poincar\'e duality, $H^i(\cG)$ is isomorphic to $H_{3-i}(\cG)$, 
and so HK holds for $\cG$. 
\end{proof}

We record here a caveat concerning the $K$-theory computations 
in \cite{MR1269302}. 
In the proof of \cite[Proposition 3]{MR1269302}, 
the kernel of the action induced by one generator 
on the coinvariants for the other generator is 
implicitly identified with the coinvariants of 
the corresponding invariant subgroup. 
Equivalently, 
the argument interchanges the invariant and coinvariant functors, 
which is not valid in general. 
As Theorem \ref{Z^2} shows, 
the natural description in dimension two is instead 
\[
K_1(C(X)\rtimes\Z^2)\cong H_1(\Z^2,C(X,\Z)), 
\]
but this does not yield the direct-sum decomposition 
used in \cite{MR1269302}. 
Thus, 
the claimed splitting of the Pimsner--Voiculescu exact sequence is not 
established by the argument given there. 
Since the three-dimensional computation in \cite[Proposition 4]{MR1269302} 
relies on this splitting, 
and the splitting asserted after (88) is likewise not justified, 
the formula stated there does not follow from the given proof. 
None of our results uses these computations. 

We next consider the class (3b), 
that is, $\Lambda=\Z^2$ and 
$K_0(\bar\theta)(\sfb_{\Z^2})=-\sfb_{\Z^2}$. 

\begin{theorem}\label{Hirsch3b}
Let $\Gamma=\Lambda\rtimes_\theta\Z$ be as above and 
let $\cG$ be the transformation groupoid 
arising from a free action of $\Gamma$ on a Cantor set $X$. 
Put $\cH:=X\rtimes\Lambda$. 
\begin{enumerate}
\item The following sequence is exact: 
\[
\xymatrix@M=8pt{
0 \ar[r] & 
H_0(\cG)\oplus\Coker(\id-H_2(\theta)) \ar[r]^-{\mu_0\oplus\nu} & 
K_0(C^*_r(\cG)) \ar[r] & 
\Ker(\id-H_1(\theta)) \ar[r] & 0, 
}
\]
where $\nu$ is a map induced by $\mu^0$ and the PV exact sequence. 
In particular, $\mu_0:H_0(\cG)\to K_0(C^*_r(\cG))$ is injective. 
\item The following sequence is split exact: 
\[
\xymatrix@M=8pt{
0 \ar[r] & 
H_1(\cG) \ar[r]^-{\mu_1} & 
K_1(C^*_r(\cG)) \ar[r] & 
H_3(\cG) \ar[r] & 0. 
}
\]
Thus, $K_1(C^*_r(\cG))$ is isomorphic to $H_1(\cG)\oplus H_3(\cG)$. 
In particular, $\mu_1:H_1(\cG)\to K_1(C^*_r(\cG))$ is injective. 
\end{enumerate}
\end{theorem}

\begin{proof}
For this class, 
we cannot make use of cohomology comparison maps for $\cG$, 
because $K_0(\bar\theta)(\sfb_{\Z^2})=-\sfb_{\Z^2}$. 
However, the same arguments as in the proofs of 
Propositions \ref{HKcommute} and \ref{coHKcommute} give 
the following sign-twisted commutative diagrams: 
\begin{equation*}
\vcenter{
\xymatrix@M=8pt{
H_1(\cH) \ar[r]^-{\id-H_1(\theta)} \ar[d]_{\mu_1} & 
H_1(\cH) \ar[d]_{\mu_1} & 
H^0(\cH) \ar[r]^-{\id+H^0(\theta)} \ar[d]_{\mu^0} & 
H^0(\cH) \ar[d]_{\mu^0} \\
K_{1}(C^*_r(\cH)) \ar[r]^-{\id-K_{1}(\bar\theta)} & 
K_{1}(C^*_r(\cH)) & 
K_{0}(C^*_r(\cH)) \ar[r]^-{\id-K_{0}(\bar\theta)} & 
K_{0}(C^*_r(\cH)) \\
}}
\end{equation*}
Moreover, 
$\mu_0\oplus\mu^0:H_0(\cH)\oplus H^0(\cH)\to K_0(C^*_r(\cH))$ 
and $\mu_1:H_1(\cH)\to K_1(C^*_r(\cH))$ are 
isomorphisms thanks to Theorem \ref{Z^2}. 

(1)\:
Combining the diagrams above 
with \eqref{HKdiagram} and \eqref{coHKdiagram}, 
we obtain the following: 
\begin{equation*}
\vcenter{
\xymatrix@M=8pt{
H_0(\cH)\oplus H^0(\cH) \ar[r]^-{\Phi_0} 
\ar[d]_{\mu_0\oplus\mu^0}^-{\cong} & 
H_0(\cH)\oplus H^0(\cH) \ar[d]_{\mu_0\oplus\mu^0}^-{\cong} & 
 & 
\phantom{H} \\
K_0(C^*_r(\cH)) \ar[r]^{\id-K_0(\bar\theta)} & 
K_0(C^*_r(\cH)) \ar[r]^-{\iota} & 
K_0(C^*_r(\cG)) \ar[r]^-{\partial} & 
\phantom{H}
}}
\end{equation*}
\begin{equation*}
\vcenter{
\xymatrix@M=8pt{
\phantom{H} & 
H_1(\cH) \ar[r]^-{\id-H_1(\theta)} \ar[d]_-{\mu_1}^{\cong} & 
H_1(\cH) \ar[d]_-{\mu_1}^{\cong} \\
\phantom{H} \ar[r]^-{\partial} & 
K_1(C^*_r(\cH)) \ar[r]^{\id-K_1(\bar\theta)} & 
K_1(C^*_r(\cH)), 
}}
\end{equation*}
where the map $\Phi_0$ is 
$(\id-H_0(\theta))\oplus(\id+H^0(\theta))$. 
Hence we have the exact sequence 
\[
\xymatrix@M=8pt{
0 \ar[r] & 
H_0(\cG)\oplus\Coker(\id+H^0(\theta)) \ar[r]^-{\mu_0\oplus{\nu}} & 
K_0(C^*_r(\cG)) \ar[r] & 
\Ker(\id-H_1(\theta)) \ar[r] & 0. 
}
\]
Since Poincar\'e duality $H^0(\cH)\cong H_2(\cH)$ intertwines 
$H^0(\theta)$ and $-H_2(\theta)$, 
this gives (1). 

(2)\:
A similar argument gives 
\[
\xymatrix@M=8pt{
0 \ar[r] & 
H_1(\cG) \ar[r]^-{\mu_1} & 
K_1(C^*_r(\cG)) \ar[r] & 
\Ker(\id+H^0(\theta)) \ar[r] & 0, 
}
\]
is exact. 
By Poincar\'e duality, 
\[
\Ker(\id+H^0(\theta))\cong\Ker(\id-H_2(\theta))
\cong H_3(\cG). 
\]
Since 
\[
\Ker(\id+H^0(\theta))
=\{f\in C(X,\Z)^\Lambda\mid f\circ\theta=-f\}
\]
is free abelian, we obtain the conclusion. 
\end{proof}

Finally, we study actions of $\Gamma$ belonging to class (3c). 
Thus, 
$\Gamma$ is a semidirect product of the Klein bottle group $\Lambda$ 
by an automorphism $\theta\in\Aut(\Lambda)$. 
When $\Lambda=\langle a,b\mid bab^{-1}=a^{-1}\rangle$, 
there exist $\ep_1,\ep_2\in\{1,-1\}$ and $k\in\Z$ 
such that $\theta(a)=a^{\ep_1}$ and $\theta(b)=b^{\ep_2}a^k$ 
(see \cite[Lemma 8.2]{IM21Invent} for instance). 

\begin{theorem}\label{Hirsch3c}
Let $\Gamma=\Lambda\rtimes_\theta\Z$ be as above and 
let $\cG$ be the transformation groupoid 
arising from a free action of $\Gamma$ on a Cantor set $X$. 
Put $\cH:=X\rtimes\Lambda$. 
\begin{enumerate}
\item The following sequence is exact: 
\[
\xymatrix@M=8pt{
0 \ar[r] & 
H_0(\cG)\oplus\Coker(\id-H_2(\theta)) \ar[r]^-{\mu_0\oplus\nu} & 
K_0(C^*_r(\cG)) \ar[r] & 
\Ker(\id-H_1(\theta)) \ar[r] & 0, 
}
\]
where $\nu$ is a map induced by $\mu^0$ and the PV exact sequence. 
In particular, 
$\mu_0:H_0(\cG)\to K_0(C^*_r(\cG))$ is injective. 
\item The following sequence is split exact: 
\[
\xymatrix@M=8pt{
0 \ar[r] & 
H_1(\cG) \ar[r]^-{\mu_1} & 
K_1(C^*_r(\cG)) \ar[r] & 
H_3(\cG) \ar[r] & 0. 
}
\]
Thus, $K_1(C^*_r(\cG))$ is isomorphic to $H_1(\cG)\oplus H_3(\cG)$. 
In particular, $\mu_1:H_1(\cG)\to K_1(C^*_r(\cG))$ is injective. 
\end{enumerate}
\end{theorem}

\begin{proof}
By Theorem \ref{Klein} (1) and Remark \ref{KleinRemark}, 
there exists a homomorphism 
$s:\Ker(\id+H^0(b))\to K_0(C^*_r(\cH))$ such that 
\[
\mu_0\oplus s:
H_0(\cH)\oplus\Ker(\id+H^0(b))\to K_0(C^*_r(\cH))
\]
is an isomorphism, 
where $H^0(b)$ denotes the automorphism of 
\[
H^0(X\rtimes\Z)=C(X,\Z)^\Z=\{f\in C(X,\Z)\mid a.f=f\}
\]
sending $f$ to $b.f$. 
As in Theorem \ref{Klein} and Remark \ref{KleinRemark}, 
we let $z:=\lambda_a\in C^*_r(\Lambda)$ and 
$u:=\lambda_b\in C^*_r(\Lambda)$. 
By Remark \ref{KleinRemark}, 
for $h=1_P-1_Q\in\Ker(\id+H^0(b))$, the unitary 
\[
v_h:=1_{P}\otimes z+1_{Q}\otimes z^*+(1_X-1_{P}-1_{Q})
\]
commutes with $u$, and 
$s(h)$ is given by 
the Bott element $\Bott(v_h,u)\in K_0(C^*_r(\cH))$. 
Hence 
\[
K_0(\bar\theta)(\Bott(v_h,u))
=\Bott(\bar\theta(v_h),\bar\theta(u))
=\Bott(v_{\bar\theta(h)}^{\ep_1},u^{\ep_2}z^k)
=\Bott(v_{\bar\theta(h)}^{\ep_1},u^{\ep_2})
=\ep_1\ep_2s(\bar\theta(h)). 
\]
Under the identification of $\Ker(\id+H^0(b))$ 
with $H_2(\cH)$, 
this is equal to $H_2(\theta)(h)$. 
Consequently, we obtain the following commutative diagram: 
\begin{equation*}
\vcenter{
\xymatrix@M=8pt{
H_0(\cH)\oplus H_2(\cH) \ar[r]^{\Phi_0} 
\ar[d]_{\mu_0\oplus s}^-{\cong} & 
H_0(\cH)\oplus H_2(\cH) \ar[d]_{\mu_0\oplus s}^-{\cong} \\
K_0(C^*_r(\cH)) \ar[r]^{\id-K_0(\bar\theta)} & 
K_0(C^*_r(\cH)), 
}}
\end{equation*}
where $\Phi_0=(\id-H_0(\theta))\oplus(\id-H_2(\theta))$. 
Combining this with the diagram \eqref{HKdiagram}, 
we obtain 
\[
\xymatrix@M=8pt{
0 \ar[r] & 
H_0(\cG)\oplus\Coker(\id-H_2(\theta)) \ar[r]^-{\mu_0\oplus\nu} & 
K_0(C^*_r(\cG)) \ar[r] & 
\Ker(\id-H_1(\theta)) \ar[r] & 0
}
\]
and 
\[
\xymatrix@M=8pt{
0 \ar[r] & 
H_1(\cG) \ar[r]^-{\mu_1} & 
K_1(C^*_r(\cG)) \ar[r] & 
H_3(\cG) \ar[r] & 0
}
\]
as desired. 
Here we use the identification 
$\Ker(\id-H_2(\theta))\cong H_3(\cG)$ 
coming from the homology long exact sequence, 
and the fact that 
$\mu_1:H_1(\cH)\to K_1(C^*_r(\cH))$ is an isomorphism 
by Theorem \ref{Klein} (2). 
\end{proof}

\begin{remark}
The homology long exact sequence in Lemma \ref{hoLES} implies that 
\[
\xymatrix@M=8pt{
0 \ar[r] & 
\Coker(\id-H_2(\theta)) \ar[r] & 
H_2(\cG) \ar[r] & 
\Ker(\id-H_1(\theta)) \ar[r] & 0
}
\]
is exact. 
In view of this, 
the exact sequences in Theorem \ref{Hirsch3b} and 
Theorem \ref{Hirsch3c} indicate that 
$K_0(C^*_r(\cG))$ is built from $H_0(\cG)$ and $H_2(\cG)$. 
\end{remark}

We conclude this subsection 
by discussing GL for poly-$\Z$ groups of Hirsch length three. 

\begin{theorem}\label{GL3}
Let $\Gamma:=\Lambda\rtimes_\theta\Z$ be 
a poly-$\Z$ group of Hirsch length three. 
Let $\cG$ be the transformation groupoid 
arising from a free action of $\Gamma$ on a Cantor set $X$. 
Then we have 
\[
\Ima D_\cG\subset\frac{1}{2}\Ima(D_\cG\circ\mu_0). 
\]
Thus GL holds for $\cG$ up to a factor of two. 
In particular, if $D_\cG(\mu_0(H_0(\cG)))$ is $2$-divisible, 
then GL holds for $\cG$. 
\end{theorem}

\begin{proof}
Set $\cH:=X\rtimes\Lambda$. 
By the preceding cases, GL holds for $\cH$. 
Consider the diagram \eqref{GLdiagram}. 
In the present setting, 
$\mu_1:H_1(\cH)\to K_1(C^*_r(\cH))$ is an isomorphism. 
Hence, for any $c\in K_0(C^*_r(\cG))$, 
there exists $h\in H_1(\cH)$ 
such that $\partial(c)=\mu_1(h)$ and $h=H_1(\theta)(h)$. 
By Proposition \ref{halfGL}, we have 
$2D_\cG(c)\in\Ima(D_\cG\circ\iota)$. 
Since GL holds for $\cH$, 
the latter is contained in $\Ima(D_\cG\circ\mu_0)$. 
Hence $2D_\cG(c)\in\Ima(D_\cG\circ\mu_0)$. 
\end{proof}

\subsection{Poly-$\Z$ groups of Hirsch length four}

In this subsection, we study actions of 
a certain class of poly-$\Z$ groups of Hirsch length four. 
Let $\Lambda:=\Z^2\rtimes\Z$ be a poly-$\Z$ group of class (3a). 
Thus, $\Lambda$ arises from an action $\Z\curvearrowright\Z^2$ 
induced by a matrix $A\in GL(2,\Z)$ with $\det A=1$. 
We choose $\sfb_\Lambda\in K_1(C^*_r(\Lambda))$ as in Section 5.2. 
Let $\Gamma$ be a semidirect product of $\Lambda$ 
by an automorphism $\theta\in\Aut(\Lambda)$. 
We assume that $K_1(\bar\theta)$ fixes $\sfb_\Lambda$, 
and choose $\sfb_\Gamma\in K_0(C^*_r(\Gamma))$ so that 
the connecting map $K_0(C^*_r(\Gamma))\to K_1(C^*_r(\Lambda))$ 
sends $\sfb_\Gamma$ to $\sfb_\Lambda$. 
Such a class $\sfb_\Gamma$ exists by exactness of the 
Pimsner--Voiculescu sequence. 
We refer to this class of groups as $(4a+)$. 

\begin{remark}
When $\Lambda\cong\Z^3$, 
the automorphism $\theta\in\Aut(\Lambda)$ is determined by 
a matrix $B\in GL(3,\Z)$. 
We may choose $\sfb_\Lambda\in K_1(C^*_r(\Lambda))$ 
so that $K_1(\bar\theta)$ fixes $\sfb_\Lambda$ 
for any $\theta$ such that $\det B=1$. 
In particular, $\Gamma=\Z^4$ fits into this framework. 
\end{remark}

\begin{theorem}\label{Hirsch4a+}
Let $\Gamma=\Lambda\rtimes_\theta\Z$ be as above and 
let $\cG$ be the transformation groupoid 
arising from a free action of $\Gamma$ on a Cantor set $X$. 
Set $\cH:=X\rtimes\Lambda$. 
\begin{enumerate}
\item The following sequence is exact: 
\[
\xymatrix@M=8pt{
0 \ar[r] & 
H_0(\cG)\oplus\Coker(\id-H_2(\theta)) \ar[r]^-{\mu_0\oplus\nu} & 
K_0(C^*_r(\cG)) \ar[r] & 
\phantom{H}
}
\]
\[
\xymatrix@M=8pt{
\phantom{H} \ar[r] & 
\Ker(\id-H_1(\theta))\oplus H_4(\cG) \ar[r] & 0, 
}
\]
where $\nu$ is a map induced by $\mu^1$ and the PV exact sequence. 
In particular, $\mu_0:H_0(\cG)\to K_0(C^*_r(\cG))$ is injective. 
\item $\mu_1\oplus\mu^1:H_1(\cG)\oplus H^1(\cG)\to K_1(C^*_r(\cG))$ 
is an isomorphism. 
\end{enumerate}
\end{theorem}

\begin{proof}
Let $\Phi_0:=(\id-H_0(\theta))\oplus(\id-H^1(\theta))$ and 
$\Phi_1:=(\id-H_1(\theta))\oplus(\id-H^0(\theta))$. 
The two diagrams \eqref{HKdiagram} and \eqref{coHKdiagram} 
combine to give the following: 
\begin{equation*}
\vcenter{
\xymatrix@M=8pt{
H_0(\cH)\oplus H^1(\cH) \ar[r]^{\Phi_0} 
\ar[d]_{\mu_0\oplus\mu^1}^-{\cong} & 
H_0(\cH)\oplus H^1(\cH) \ar[d]_{\mu_0\oplus\mu^1}^-{\cong} & 
 & 
\phantom{H} \\
K_0(C^*_r(\cH)) \ar[r]^{\id-K_0(\bar\theta)} & 
K_0(C^*_r(\cH)) \ar[r]^-{\iota} & 
K_0(C^*_r(\cG)) \ar[r]^-{\partial} & 
\phantom{H}
}}
\end{equation*}
\begin{equation*}
\vcenter{
\xymatrix@M=8pt{
\phantom{H} & 
H_1(\cH)\oplus H^0(\cH) \ar[r]^-{\Phi_1} 
\ar[d]_-{\mu_1\oplus\mu^0}^{\cong} & 
H_1(\cH)\oplus H^0(\cH) \ar[d]_-{\mu_1\oplus\mu^0}^{\cong} \\
\phantom{H} \ar[r]^-{\partial} & 
K_1(C^*_r(\cH)) \ar[r]^{\id-K_1(\bar\theta)} & 
K_1(C^*_r(\cH)). 
}}
\end{equation*}
By Theorem \ref{Hirsch3a}, 
$\mu_0\oplus\mu^1:H_0(\cH)\oplus H^1(\cH)\to K_0(C^*_r(\cH))$ and 
$\mu_1\oplus\mu^0:H_1(\cH)\oplus H^0(\cH)\to K_1(C^*_r(\cH))$ are 
isomorphisms. 
Therefore, we obtain the following exact sequence: 
\[
\xymatrix@M=8pt{
0 \ar[r] & 
H_0(\cG)\oplus\Coker(\id-H^1(\theta)) \ar[r]^-{\mu_0\oplus\nu} & 
K_0(C^*_r(\cG)) \ar[r] & 
}
\]
\[
\xymatrix@M=8pt{
\ar[r] & 
\Ker(\id-H_1(\theta))\oplus\Ker(\id-H^0(\theta)) \ar[r] & 
0. 
}
\]
Since Poincar\'e duality $H_i(\cH)\cong H^{3-i}(\cH)$ intertwines 
$H_i(\theta)$ and $H^{3-i}(\theta)$, 
the sequence above identifies with the sequence in (1). 
Here $\Ker(\id-H_3(\theta))$ is identified with $H_4(\cG)$ 
through the homology long exact sequence. 
The assertion (2) follows from the odd-degree part of the same diagrams. 
\end{proof}

Next we investigate the orientation reversing case. 
Let $\Lambda:=\Z^2\rtimes\Z$ be a poly-$\Z$ group of class (3a). 
Thus, $\Lambda$ arises from an action $\Z\curvearrowright\Z^2$ 
induced by a matrix $A\in GL(2,\Z)$ with $\det A=1$. 
We choose $\sfb_\Lambda\in K_1(C^*_r(\Lambda))$ as in Section 5.2. 
Let $\Gamma$ be a semidirect product of $\Lambda$ 
by an automorphism $\theta\in\Aut(\Lambda)$ 
such that $K_1(\bar\theta)(\sfb_\Lambda)=-\sfb_\Lambda$. 
We refer to this class of groups as $(4a-)$. 

\begin{remark}
When $\Lambda\cong\Z^3$, 
the automorphism $\theta\in\Aut(\Lambda)$ is determined by 
a matrix $B\in GL(3,\Z)$. 
We may choose $\sfb_\Lambda\in K_1(C^*_r(\Lambda))$ 
so that $K_1(\bar\theta)(\sfb_\Lambda)=-\sfb_\Lambda$ 
for any $\theta$ such that $\det B=-1$. 
\end{remark}

\begin{theorem}\label{Hirsch4a-}
Let $\Gamma=\Lambda\rtimes_\theta\Z$ be as above and 
let $\cG$ be the transformation groupoid 
arising from a free action of $\Gamma$ on a Cantor set $X$. 
Set $\cH:=X\rtimes\Lambda$. 
\begin{enumerate}
\item The following sequence is exact: 
\[
\xymatrix@M=8pt{
0 \ar[r] & 
H_0(\cG)\oplus\Coker(\id-H_2(\theta)) \ar[r]^-{\mu_0\oplus\nu^1} & 
K_0(C^*_r(\cG)) \ar[r] & 
}
\]
\[
\xymatrix@M=8pt{
\ar[r] & 
\Ker(\id-H_1(\theta))\oplus H_4(\cG) \ar[r] & 0, 
}
\]
where $\nu^1$ is a map induced by $\mu^1$ and the PV exact sequence. 
In particular, $\mu_0:H_0(\cG)\to K_0(C^*_r(\cG))$ is injective. 
\item The following sequence is exact: 
\[
\xymatrix@M=8pt{
0 \ar[r] & 
H_1(\cG)\oplus\Coker(\id-H_3(\theta)) \ar[r]^-{\mu_1\oplus\nu^0} & 
K_1(C^*_r(\cG)) \ar[r] & 
\Ker(\id-H_2(\theta)) \ar[r] & 0, 
}
\]
where $\nu^0$ is a map induced by $\mu^0$ and the PV exact sequence. 
In particular, $\mu_1:H_1(\cG)\to K_1(C^*_r(\cG))$ is injective. 
\end{enumerate}
\end{theorem}

\begin{proof}
The argument is the same as that for Theorem \ref{Hirsch4a+}, 
except that the cohomology comparison maps acquire a sign. 
By the proof of Proposition \ref{coHKcommute}, 
we know that $\mu^i:H^i(\cH)\to K_{i+1}(C^*_r(\cH))$ intertwines 
$H^i(\theta)$ and $-K_{i+1}(\bar\theta)$. 
Combining the two diagrams \eqref{HKdiagram} and \eqref{coHKdiagram} 
yields the following: 
\begin{equation*}
\vcenter{
\xymatrix@M=8pt{
H_0(\cH)\oplus H^1(\cH) \ar[r]^{\Phi_0} 
\ar[d]_{\mu_0\oplus\mu^1}^-{\cong} & 
H_0(\cH)\oplus H^1(\cH) \ar[d]_{\mu_0\oplus\mu^1}^-{\cong} & 
 & 
\phantom{H} \\
K_0(C^*_r(\cH)) \ar[r]^{\id-K_0(\bar\theta)} & 
K_0(C^*_r(\cH)) \ar[r]^-{\iota} & 
K_0(C^*_r(\cG)) \ar[r]^-{\partial} & 
\phantom{H}
}}
\end{equation*}
\begin{equation*}
\vcenter{
\xymatrix@M=8pt{
\phantom{H} & 
H_1(\cH)\oplus H^0(\cH) \ar[r]^-{\Phi_1} 
\ar[d]_-{\mu_1\oplus\mu^0}^{\cong} & 
H_1(\cH)\oplus H^0(\cH) \ar[d]_-{\mu_1\oplus\mu^0}^{\cong} \\
\phantom{H} \ar[r]^-{\partial} & 
K_1(C^*_r(\cH)) \ar[r]^{\id-K_1(\bar\theta)} & 
K_1(C^*_r(\cH)). 
}}
\end{equation*}
where the map $\Phi_0$ is 
$(\id-H_0(\theta))\oplus(\id+H^1(\theta))$ and 
the map $\Phi_1$ is $(\id-H_1(\theta))\oplus(\id+H^0(\theta))$. 
By Theorem \ref{Hirsch3a}, 
$\mu_0\oplus\mu^1:H_0(\cH)\oplus H^1(\cH)\to K_0(C^*_r(\cH))$ and 
$\mu_1\oplus\mu^0:H_1(\cH)\oplus H^0(\cH)\to K_1(C^*_r(\cH))$ are 
isomorphisms. 
It follows that we obtain the following exact sequence: 
\[
\xymatrix@M=8pt{
0 \ar[r] & 
H_0(\cG)\oplus\Coker(\id+H^1(\theta)) \ar[r]^-{\mu_0\oplus\nu^1} & 
K_0(C^*_r(\cG)) \ar[r] & 
}
\]
\[
\xymatrix@M=8pt{
\ar[r] & 
\Ker(\id-H_1(\theta))\oplus\Ker(\id+H^0(\theta)) \ar[r] & 
0. 
}
\]
Since Poincar\'e duality $H_i(\cH)\cong H^{3-i}(\cH)$ intertwines 
$H_i(\theta)$ and $-H^{3-i}(\theta)$, 
the sequence above identifies with the sequence in (1). 
Here $\Ker(\id-H_3(\theta))$ is identified with $H_4(\cG)$ 
through the homology long exact sequence. 

For $K_1(C^*_r(\cG))$, the same argument gives 
\[
\xymatrix@M=8pt{
0 \ar[r] & 
H_1(\cG)\oplus\Coker(\id+H^0(\theta)) \ar[r]^-{\mu_1\oplus\nu^0} & 
K_1(C^*_r(\cG)) \ar[r] & 
\Ker(\id+H^1(\theta)) \ar[r] & 0. 
}
\]
Under Poincar\'e duality, 
this identifies with the sequence in (2). 
\end{proof}

\begin{remark}
The homology long exact sequence in Lemma \ref{hoLES} implies that 
\[
\xymatrix@M=8pt{
0 \ar[r] & 
\Coker(\id-H_*(\theta)) \ar[r] & 
H_*(\cG) \ar[r] & 
\Ker(\id-H_{*-1}(\theta)) \ar[r] & 0
}
\]
is exact. 
In view of this, the exact sequence 
in Theorem \ref{Hirsch4a+} (1) and Theorem \ref{Hirsch4a-} (1) 
indicates that $K_0(C^*_r(\cG))$ is built 
from $H_0(\cG)$, $H_2(\cG)$ and $H_4(\cG)$. 
Moreover $H_4(\cG)$ is free, and so 
it is a direct summand of $K_0(C^*_r(\cG))$. 
In addition, 
the exact sequence in Theorem \ref{Hirsch4a-} (2) indicates that 
$K_1(C^*_r(\cG))$ is built from $H_1(\cG)$ and $H_3(\cG)$. 
\end{remark}

In the remainder of this subsection, 
we discuss GL for poly-$\Z$ groups of Hirsch length four. 
Let $\Lambda$ be a poly-$\Z$ group of Hirsch length three 
and let $\theta\in\Aut(\Lambda)$ be an automorphism. 
Set $\Gamma:=\Lambda\rtimes_\theta\Z$. 

\begin{theorem}\label{GL4}
Assume that one of the following conditions holds. 
\begin{enumerate}
\item $\Lambda$ belongs to class $(3a)$ and 
$\theta\in\Aut(\Lambda)$ is the identity. 
\item $\Gamma$ belongs to class $(4a-)$, i.e. 
$\Lambda$ belongs to class $(3a)$ and 
$K_1(\bar\theta)(\sfb_\Lambda)=-\sfb_\Lambda$. 
\end{enumerate}
When $\cG:=X\rtimes\Gamma$ is a transformation groupoid 
of a free action of $\Gamma$ on a Cantor set $X$, 
one has 
\[
\Ima D_\cG\subset\frac{1}{2}\Ima(D_\cG\circ\mu_0). 
\]
Thus GL holds for $\cG$ up to a factor of two. 
In particular, if $D_\cG(\mu_0(H_0(\cG)))$ is $2$-divisible, 
then GL holds for $\cG$. 
\end{theorem}

\begin{proof}
Set $\cH:=X\rtimes\Lambda$. 
We first prove (2). 
By Theorem \ref{Hirsch3a}, 
$\mu_1\oplus\mu^0:H_1(\cH)\oplus H^0(\cH)\to K_1(C^*_r(\cH))$ 
is an isomorphism. 
Consider the commutative diagram 
\begin{equation}\label{GLdiagram2}
\vcenter{
\xymatrix@M=8pt{
&&
H_1(\cH)\oplus H^0(\cH) \ar[r]^-{\Phi_1} 
\ar[d]_-{\mu_1\oplus\mu^0}^-{\cong} & 
H_1(\cH)\oplus H^0(\cH) \ar[d]_-{\mu_1\oplus\mu^0}^-{\cong} \\
K_0(C^*_r(\cH)) \ar[r]^-{\iota} & 
K_0(C^*_r(\cG)) \ar[r]^-{\partial} & 
K_1(C^*_r(\cH)) \ar[r]^-{\id-K_1(\bar\theta)} & 
K_1(C^*_r(\cH)), 
}}
\end{equation}
where the map $\Phi_1$ is 
$(\id-H_1(\theta))\oplus(\id+H^0(\theta))$. 
Theorem \ref{GL3} tells us that 
\[
\Ima D_{\cH}\subset\frac{1}{2}\Ima(D_\cH\circ\mu_0). 
\]
Let $c\in K_0(C^*_r(\cG))$. 
Since $\partial(c)$ belongs to $\Ker(\id-K_1(\bar\theta))$, 
the diagram above and the isomorphism 
$\mu_1\oplus\mu^0$ give elements $h_1\in H_1(\cH)$ and 
$h_0\in H^0(\cH)$ such that
\[
\partial(c)=\mu_1(h_1)+\mu^0(h_0),\quad 
h_1=H_1(\theta)(h_1),\quad 
h_0=-H^0(\theta)(h_0). 
\]
By exactness, the two summands have lifts to 
$K_0(C^*_r(\cG))$, say $c_1$ for $\mu_1(h_1)$ and 
$c_0$ for $\mu^0(h_0)$, and 
the remaining difference lies in $\Ima\iota$. 
By Proposition \ref{halfGL}, 
\[
D_\cG(c_1)\in\frac{1}{2}\Ima(D_\cG\circ\mu_0)+\Ima(D_\cG\circ\iota)
\subset\frac{1}{2}\Ima(D_\cG\circ\mu_0). 
\]
The element $h_0$ may be regarded as 
a $\Lambda$-invariant function on $X$ 
satisfying $h_0\circ\theta=-h_0$. 
There exist $\Lambda$-invariant clopen subsets $P_1,P_2,\dots,P_m$ 
and $Q_1,Q_2,\dots,Q_m$ of $X$ such that 
$P_j\cap Q_j=\emptyset$, $\theta(P_j)=Q_j$, $\theta(Q_j)=P_j$ 
and 
\[
h_0=\sum_{j=1}^m1_{P_j}-\sum_{j=1}^m1_{Q_j}. 
\]
Choose a unitary $u\in M_n(C^*_r(\Lambda))$ 
representing $\sfb_\Lambda\in K_1(C^*_r(\Lambda))$. 
After stabilizing if necessary, 
we may assume that $u\bar\theta(u)$ is homotopic 
to the identity in $M_n(C^*_r(\Lambda))$. 
Let $\gamma:[0,1]\to U(M_n(C^*_r(\Lambda)))$ be 
a piecewise smooth path from $1$ to $u\bar\theta(u)$. 
For each $j$, set 
\[
v_j:=1_{P_j}\otimes u+1_{Q_j}\otimes u^*+(1_X-1_{P_j}-1_{Q_j})
\in C(X)^\Lambda\otimes M_n(C^*_r(\Lambda))\subset C^*_r(\cH). 
\]
Then $\mu^0(h_0)=\sum_jK_1(v_j)$, and 
\[
v_j\bar\theta(v_j^*)=1_{P_j}\otimes u\bar\theta(u)
+1_{Q_j}\otimes u^*\bar\theta(u^*)+(1_X-1_{P_j}-1_{Q_j}). 
\]
Define $\gamma_j:[0,1]
\to U(C(X)^\Lambda\otimes M_n(C^*_r(\Lambda)))$ by 
\[
\gamma_j(t)
:=1_{P_j}\otimes\gamma(t)
+1_{Q_j}\otimes u^*\gamma(t)^*u+(1_X-1_{P_j}-1_{Q_j}). 
\]
The determinant along this path vanishes on traces from $M(\cG)$, 
that is, $d_{\gamma_j}(\tau_\nu|_{C^*_r(\cH)})=0$ for any $\nu\in M(\cG)$. 
Indeed, since $\nu$ is $\Gamma$-invariant and $\theta(P_j)=Q_j$, 
we have $\nu(P_j)=\nu(Q_j)$, and 
the two summands in the definition of $\gamma_j$ contribute 
with opposite signs. 
By the trace formula for the Pimsner--Voiculescu boundary map, 
the determinant term associated with $v_j\bar\theta(v_j^*)$ is 
represented by $d_{\gamma_j}$, 
which vanishes on all traces coming from $M(\cG)$. 
Hence the corresponding contribution to $D_\cG(c_0)$ 
lies in $\Ima(D_\cG\circ\iota)$. 
This proves (2). 

The proof of (1) is largely the same as that of (2). 
The difference lies in the handling of $h_0$. 
By assumption, 
the map $\Phi_1$ in the diagram \eqref{GLdiagram2} becomes 
$(\id-H_1(\theta))\oplus(\id-H^0(\theta))$, 
and $\partial(c)$ is decomposed as 
\[
\partial(c)=\mu_1(h_1)+\mu^0(h_0),\quad 
h_1=H_1(\theta)(h_1),\quad 
h_0=H^0(\theta)(h_0). 
\]
The element $h_0$ may be regarded as 
a $\Gamma$-invariant function on $X$. 
Since such a function is a finite integral linear combination
of indicator functions of $\Gamma$-invariant clopen subsets, 
it suffices to consider 
a single $\Gamma$-invariant clopen subset $P$ and assume $h_0=1_P$. 
Choose a unitary $u\in M_n(C^*_r(\Lambda))$ 
representing $\sfb_\Lambda\in K_1(C^*_r(\Lambda))$. 
In case (1), $\theta=\id$, and hence 
we have $u=\bar\theta(u)$. 
Put 
\[
v:=1_{P}\otimes u+(1_X-1_{P})
\in C(X)^\Lambda\otimes M_n(C^*_r(\Lambda))\subset C^*_r(\cH). 
\]
Then $\mu^0(1_P)=K_1(v)$, and $v\bar\theta(v^*)=1$. 
Therefore the determinant term is zero 
on all traces coming from $M(\cG)$. 
Hence the corresponding contribution to $D_\cG(c_0)$ 
lies in $\Ima(D_\cG\circ\iota)$. 
Since Theorem \ref{GL3} applied to $\cH$ gives 
\[
\Ima(D_\cG\circ\iota)\subset\frac{1}{2}\Ima(D_\cG\circ\mu_0), 
\]
we obtain the desired inclusion.
\end{proof}

\begin{remark}
The group $\Z^4$ is covered by Theorem \ref{GL4} (1). 
It is not clear 
whether the hypothesis $\theta=\id$ can be weakened 
to $K_1(\bar\theta)(\sfb_\Lambda)=\sfb_\Lambda$. 
The argument above would still apply 
if one could choose a unitary $u\in M_n(C^*_r(\Lambda))$ 
representing $\sfb_\Lambda$ 
such that the de la Harpe--Skandalis determinant of 
$u\bar\theta(u^*)$ vanishes, 
at least after evaluation on the traces arising from $M(\cG)$. 
We do not know when such a choice is possible. 
\end{remark}

\subsection{Poly-$\Z$ groups of Hirsch length five}

For Hirsch length five, our method currently gives only 
limited information. 
The main obstruction is our limited control of 
poly-$\Z$ groups of Hirsch length four. 
We therefore restrict attention to the following special case 
and prove only the injectivity of $\mu_1$. 

Let $\Lambda$ be a poly-$\Z$ group of class $(4a+)$ 
treated in Section 5.3. 
We choose $\sfb_\Lambda\in K_0(C^*_r(\Lambda))$ as in Section 5.3. 
Let $\Gamma$ be a semidirect product of $\Lambda$ 
by an automorphism $\theta\in\Aut(\Lambda)$. 
We assume that $K_0(\bar\theta)$ fixes $\sfb_\Lambda$, 
and choose $\sfb_\Gamma\in K_1(C^*_r(\Gamma))$ so that 
the connecting map $K_1(C^*_r(\Gamma))\to K_0(C^*_r(\Lambda))$ 
sends $\sfb_\Gamma$ to $\sfb_\Lambda$. 

\begin{remark}
When $\Lambda\cong\Z^4$, 
the automorphism $\theta\in\Aut(\Lambda)$ is determined by 
a matrix $B\in GL(4,\Z)$. 
We may choose $\sfb_\Lambda\in K_0(C^*_r(\Lambda))$ 
so that $K_0(\bar\theta)$ fixes $\sfb_\Lambda$ 
for any $\theta$ such that $\det B=1$. 
In particular, $\Gamma=\Z^5$ fits into this framework. 
\end{remark}

\begin{theorem}\label{Hirsch5a++}
Let $\Gamma=\Lambda\rtimes_\theta\Z$ be as above and 
let $\cG$ be the transformation groupoid 
arising from a free action of $\Gamma$ on a Cantor set $X$. 
Then $\mu_1:H_1(\cG)\to K_1(C^*_r(\cG))$ is injective. 
\end{theorem}

\begin{proof}
Set $\cH:=X\rtimes\Lambda$. 
Suppose that $g\in H_1(\cG)$ is in the kernel of $\mu_1$. 
In the diagram \eqref{HKdiagram}, 
$\mu_0:H_0(\cH)\to K_0(C^*_r(\cH))$ is injective and 
$\mu_1:H_1(\cH)\to K_1(C^*_r(\cH))$ is injective 
by Theorem \ref{Hirsch4a+}. 
A diagram chase in \eqref{HKdiagram} gives 
$h\in H_1(\cH)$ whose image in $H_1(\cG)$ is $g$ 
and $k\in K_1(C^*_r(\cH))$ such that 
$(\id-K_1(\bar\theta))(k)=\mu_1(h)$. 
On the other hand, 
Propositions \ref{HKcommute} and \ref{coHKcommute} give the diagram 
\begin{equation*}
\vcenter{
\xymatrix@M=8pt{
H_1(\cH)\oplus H^1(\cH) \ar[r]^-{\Phi_1} 
\ar[d]_{\mu_1\oplus\mu^1}^-{\cong} & 
H_1(\cH)\oplus H^1(\cH) \ar[d]_{\mu_1\oplus\mu^1}^-{\cong} \\
K_{1}(C^*_r(\cH)) \ar[r]^-{\id-K_{1}(\bar\theta)} & 
K_{1}(C^*_r(\cH)), 
}}
\end{equation*}
where $\Phi_1:=(\id-H_1(\theta))\oplus(\id-H^1(\theta))$, 
is commutative, 
and the vertical map $\mu_1\oplus\mu^1$ is an isomorphism 
thanks to Theorem \ref{Hirsch4a+} (2). 
Hence $h$ is in the image of $\id-H_1(\theta)$, 
and therefore $g=0$ by exactness of the homology sequence. 
\end{proof}

\section{Cohomology comparison maps for mapping tori}

In this section, 
we introduce cohomology comparison maps 
for mapping tori arising from actions of $\Z^N$. 
We let $e_1,e_2,\dots,e_N$ denote 
the canonical generators of $\Z^N$. 
For an action of $\Z^N$ on a Cantor set $X$, 
its mapping torus $\cM(\Z^N\curvearrowright X)$ is defined as 
\[
\cM(\Z^N\curvearrowright X):=
\{f\in C(\R^N\times X)\mid f(t{+}e_i,e_i.x)=f(t,x)
\quad\forall (t,x)\in\R^N\times X,\ 1\leq i\leq N\}. 
\]
We identify $\cM(\Z^N\curvearrowright X)$ 
as a subalgebra of $C([0,1]^N,C(X))$. 
It is standard that 
$K_i(\cM(\Z^N\curvearrowright X))$ is isomorphic to 
$K_{i+N}(C(X)\rtimes\Z^N)$. 

Let $\cG$ be the transformation groupoid of 
$\Z^N\curvearrowright X$. 
We define the cohomology comparison maps 
$\tilde\mu^i:H^i(\cG)\to K_i(\cM(\Z^N\curvearrowright X))$ 
for mapping tori as follows. 
This is essentially the construction described in the proof of 
\cite[Theorem 2]{BKL01CRASParis}, 
but we explain it for completeness. 
There exists a canonical map 
$H^0(\cG)\cong C(X,\Z)^{\Z^N}\to K_0(C(X)^{\Z^N})$. 
Since $C(X)^{\Z^N}$ can be identified 
with a subalgebra of $\cM(\Z^N\curvearrowright X)$, 
we have a homomorphism 
from $K_0(C(X)^{\Z^N})$ to $K_0(\cM(\Z^N\curvearrowright X))$. 
Let $\tilde\mu^0:H^0(\cG)\to K_0(\cM(\Z^N\curvearrowright X))$ 
be the composition of these homomorphisms. 

We next define 
$\tilde\mu^1:H^1(\cG)\to K_1(\cM(\Z^N\curvearrowright X))$. 
Take $\xi\in\Hom(\cG,\Z)$. 
For each $j=1,2,\dots,N$, 
set $f_j:=\xi(\cdot,e_j)\in C(X,\Z)$. 
The cocycle identity gives $d_j(f_k)=d_k(f_j)$ for any $j,k$, 
where $d_j(f):=e_j.f-f$. 
Define a unitary $u_\xi\in C([0,1]^N,C(X))$ by 
\[
u_\xi(t_1,t_2,\dots,t_N):=
\exp2\pi i\left(\sum_{\emptyset\neq J\subset\{1,2,\dots,N\}}
f_J\prod_{j\in J}t_j\right)
\]
with $f_{\{j_1,j_2,\dots,j_k\}}:=
(d_{j_1}\circ\dots\circ d_{j_{k-1}})(f_{j_k})$ 
(this does not depend on the order of $j_1,j_2,\dots,j_k$). 
The defining relations imply that 
$u_\xi$ is a unitary in $\cM(\Z^N\curvearrowright X)$, 
and we set 
\[
\tilde\mu^1([\xi]):=K_1(u_\xi)
\in K_1(\cM(\Z^N\curvearrowright X)). 
\]
One checks, as in \cite[Theorem 2]{BKL01CRASParis}, that 
the assignment $\xi\mapsto K_1(u_\xi)$ is additive and 
vanishes on coboundaries. 
Hence it descends to $H^1(\cG)$. 

Consider the subgroup $\Z^{N-1}\subset\Z^N$ 
generated by $e_1,e_2,\dots,e_{N-1}$, 
and let $\cH:=X\rtimes\Z^{N-1}$ be 
the associated transformation groupoid. 
The last generator $e_N$ induces an automorphism $\theta$ of $\cH$ 
so that $\cG=\cH\rtimes_\theta\Z$. 
We also define 
$\bar\theta\in\Aut(\cM(\Z^{N-1}\curvearrowright X))$ by 
\[
\bar\theta(f)(t_1,t_2,\dots,t_{N-1}):=
e_N.(f(t_1,t_2,\dots,t_{N-1})). 
\]
Then the mapping torus of $\bar\theta$ is canonically 
isomorphic to $\cM(\Z^{N}\curvearrowright X)$. 

Consider the following diagram: 
\begin{equation*}
\vcenter{
\xymatrix@M=8pt{
0 \ar[r] & 
H^0(\cG) \ar[r] \ar[d]_{\tilde\mu^0} \ar@{}[dr]|-{(1)} & 
H^0(\cH) \ar[r]^-{\id-H^0(\theta)} \ar[d]_{\tilde\mu^0} 
\ar@{}[dr]|-{(2)} & 
H^0(\cH) \ar[r]^-{\delta} \ar[d]_{\tilde\mu^0} \ar@{}[dr]|-{(3)} & 
\phantom{H^1(\cG)} \\
\ar[r] & 
K_{0}(\cM(\Z^N)) \ar[r]  & 
K_{0}(\cM(\Z^{N-1})) \ar[r]^-{\id-K_{0}(\bar\theta)} & 
K_{0}(\cM(\Z^{N-1})) \ar[r]^-{\partial} & \phantom{H^1(\cG)}
}}
\end{equation*}
\begin{equation}\label{coMTdiagram}
\vcenter{
\xymatrix@M=8pt{
\phantom{H} \ar[r]^-{\delta} \ar@{}[dr]|-{(3)} & 
H^1(\cG) \ar[r] \ar[d]_{\tilde\mu^1} \ar@{}[dr]|-{(4)} &
H^1(\cH) \ar[r]^-{\id-H^1(\theta)} \ar[d]_{\tilde\mu^1} 
\ar@{}[dr]|-{(5)} & 
H^1(\cH) \ar[d]_{\tilde\mu^1} \ar[r] & \\
\phantom{H} \ar[r]^-{\partial} & 
K_{1}(\cM(\Z^N)) \ar[r] & 
K_{1}(\cM(\Z^{N-1})) \ar[r]^-{\id-K_{1}(\bar\theta)} & 
K_{1}(\cM(\Z^{N-1})) \ar[r] & 
}}
\end{equation}
where the top horizontal sequence is 
a part of the long exact sequence in Lemma \ref{coLES}, 
the bottom horizontal sequence is 
the exact sequence for the mapping torus. 
Here and below, the mapping tori are denoted by 
$\cM(\Z^N)$ and $\cM(\Z^{N-1})$ for short. 

\begin{proposition}\label{coMTcommute}
The diagram \eqref{coMTdiagram} is commutative. 
\end{proposition}

\begin{proof}
The squares (1) and (2) commute by construction. 

We next verify (3). 
Take $g\in C(X,\Z)^{\Z^{N-1}}$. 
There is a homomorphism 
$\xi:\cG=X\times\Z^{N-1}\times\Z\to\Z$ satisfying 
\[
\xi(x,l,0)=0\quad\text{and}\quad 
\xi(x,l,1)=g(x)\quad\forall (x,l)\in\cH=X\times\Z^{N-1}, 
\]
and $\delta(g)$ is equal to 
the equivalence class of $\xi$ in $H^1(\cG)$. 
By definition, 
$\tilde\mu^1([\xi])$ is given by the $K_1$-class of the unitary 
\[
u_\xi(t_1,t_2,\dots,t_N)=\exp2\pi i g t_N
\]
in $\cM(\Z^N\curvearrowright X)$. 
Hence we get $\tilde\mu^1(\delta(g))=\tilde\mu^1([\xi])
=K_1(u_\xi)=\partial(\tilde\mu^0(g))$ as desired. 

In order to check that (4) is commutative, 
we pick $\xi$ from $\Hom(\cG,\Z)$. 
The image of $[\xi]$ by the map $H^1(\cG)\to H^1(\cH)$ 
is the restriction $\eta:=\xi|\cH$. 
For $j=1,2,\dots,N$, we put $f_j:=\xi(\cdot,e_j)\in C(X,\Z)$. 
By definition, 
$\tilde\mu^1([\xi])$ is given by the $K_1$-class of the unitary 
\[
u_\xi(t_1,t_2,\dots,t_N):=
\exp2\pi i\left(\sum_{\emptyset\neq J\subset\{1,2,\dots,N\}}
f_J\prod_{j\in J}t_j\right)
\]
with $f_{\{j_1,j_2,\dots,j_k\}}:=
(d_{j_1}\circ\dots\circ d_{j_{k-1}})(f_{j_k})$. 
Substituting $0$ into $t_N$, one obtains 
\[
u_\xi(t_1,t_2,\dots,t_{N-1},0)=
\exp2\pi i\left(\sum_{\emptyset\neq J\subset\{1,2,\dots,N-1\}}
f_J\prod_{j\in J}t_j\right), 
\]
whose restriction to $[0,1]^{N-1}$ is equal to $u_\eta$. 
It follows that (4) is commutative. 

The square (5) commutes by construction. 
\end{proof}

The diagram \eqref{coMTdiagram} has essentially 
the same information as the diagram \eqref{coHKdiagram}. 
In particular, 
even if we use \eqref{coMTdiagram} instead of \eqref{coHKdiagram}, 
we can only reach the same conclusion regarding HK as in Section 5. 
As for GL, however, we may improve the conclusion 
thanks to the following theorem, which is due to 
Bellissard, Kellendonk and Legrand \cite{BKL01CRASParis}. 
Although the proof appears there, 
we include it here for the reader's convenience. 

\begin{lemma}[{\cite[Theorem 2]{BKL01CRASParis}}]\label{Connes}
Suppose that $N\geq3$ is odd. 
For any $\nu\in M(X\rtimes\Z^N)$ and $h\in H^1(\cG)$, 
we have $\langle[\hat\tau_\nu],\tilde\mu^1(h)\rangle=0$, 
where $\hat\tau_\nu$ is 
the cyclic $N$-cocycle on $\cM(\Z^N\curvearrowright X)$ 
arising from the tracial state $\tau_\nu$ 
and $\langle[\hat\tau_\nu],\cdot\rangle$ is the Connes pairing. 
\end{lemma}

\begin{proof}
Take $\xi\in\Hom(\cG,\Z)$ such that $h=[\xi]$. 
By definition, $\tilde\mu^1(h)=K_1(u_\xi)$. 
Let $\delta_i$ be the derivation associated with 
the coordinate $t_i$ on $\cM(\Z^N\curvearrowright X)$. 
Then, 
\begin{align*}
& \langle[\hat\tau_\nu],\tilde\mu^1(h)\rangle \\
&=\sum_{\sigma\in S_N}\sgn(\sigma)
\int_{[0,1]^N}\tau_\nu\left(u_\xi^*\delta_{\sigma(1)}(u_\xi)
\delta_{\sigma(2)}(u_\xi^*)\dots
\delta_{\sigma(N{-}1)}(u_\xi^*)\delta_{\sigma(N)}(u_\xi)\right)\ 
dt_1dt_2\dots dt_N. 
\end{align*}
Since the algebra is commutative, 
the term corresponding to $\sigma\in S_N$ is unchanged 
when $\sigma(1)$ and $\sigma(3)$ are interchanged, 
whereas its sign is reversed. 
Hence the terms cancel in pairs, and the sum is zero. 
\end{proof}

The lemma above immediately implies the following. 

\begin{theorem}[{\cite{BKL01CRASParis}}]
When $\cG$ arises from a free action of $\Z^3$ on a Cantor set $X$, 
GL holds for $\cG$. 
\end{theorem}

\begin{proof}
By Theorem \ref{Hirsch3a}, 
$\mu_0\oplus\mu^1:H_0(\cG)\oplus H^1(\cG)\to K_0(C^*_r(\cG))$ 
is an isomorphism. 
As mentioned above, we may use $\tilde\mu^1$ instead of $\mu^1$ 
under the identification of 
$K_1(\cM(\Z^3\curvearrowright X))$ with $K_0(C^*_r(\cG))$. 
Under this identification, 
the Connes pairing with $[\hat\tau_\nu]$ agrees, 
up to the usual sign convention, 
with the trace pairing $D_\cG(\cdot)(\nu)$. 
It follows from the lemma above that 
the image of $\tilde\mu^1$, under the above identification, 
is contained in the kernel of $D_\cG$. 
Hence GL holds for $\cG$. 
\end{proof}

We conclude this section by considering GL for $\Z^5$. 
To do so, assume that $N\geq 5$ is odd and 
consider the commutative diagram: 
\begin{equation}\label{GLdiagram3}
\vcenter{
\xymatrix@M=8pt{
&&
H_1(\cH) \ar[r]^-{\id-H_1(\theta)} 
\ar[d]_-{\mu_1} & 
H_1(\cH) \ar[d]_-{\mu_1} \\
K_0(C^*_r(\cH)) \ar[r]^-{\iota} \ar[d]^-{\cong} & 
K_0(C^*_r(\cG)) \ar[r] \ar[d]^-{\cong} & 
K_1(C^*_r(\cH)) \ar[r]^-{\id-K_1(\bar\theta)} \ar[d]^-{\cong} & 
K_1(C^*_r(\cH)) \ar[d]^-{\cong} \\
K_0(\cM(\Z^{N-1})) \ar[r] & 
K_1(\cM(\Z^N)) \ar[r]^-{q} & 
K_1(\cM(\Z^{N-1})) \ar[r]^-{\id-K_1(\bar\theta)} & 
K_1(\cM(\Z^{N-1})) \\
&&
H^1(\cH) \ar[r]^-{\id-H^1(\theta)} 
\ar[u]_-{\tilde\mu^1} & 
H^1(\cH) \ar[u]_-{\tilde\mu^1} 
}}
\end{equation}
Here $q$ denotes the map induced by evaluation at $t_N=0$. 

\begin{lemma}
For any $h\in\Ker(\id-H^1(\theta))$, 
there exists a unitary $w\in\cM(\Z^N\curvearrowright X)$ 
such that $K_1(q(w))=\tilde\mu^1(h)$. 
\end{lemma}

\begin{proof}
Choose $\eta\in\Hom(\cH,\Z)$ such that $h=[\eta]$. 
For $j=1,2,\dots,N{-}1$, set $f_j:=\eta(\cdot,e_j)\in C(X,\Z)$. 
By definition, 
$\tilde\mu^1([\eta])$ is given by the $K_1$-class of the unitary 
\[
u_\eta(t_1,t_2,\dots,t_{N-1}):=
\exp2\pi i\left(\sum_{\emptyset\neq J\subset\{1,2,\dots,N-1\}}
f_J\prod_{j\in J}t_j\right)
\]
with $f_{\{j_1,j_2,\dots,j_k\}}:=
(d_{j_1}\circ\dots\circ d_{j_{k-1}})(f_{j_k})$. 
Since $\eta-\eta\circ\theta^{-1}$ is a coboundary, 
there exists $g\in C(X,\Z)$ such that 
\[
(\eta-\eta\circ\theta^{-1})(x,e_j)=(g-e_j.g)(x)
\]
holds for any $x\in X$ and $j=1,2,\dots,N{-}1$. 
Define a unitary $w\in C([0,1]^N,C(X))$ by 
\[
w(t_1,t_2,\dots,t_N):=
\exp2\pi i\left(gt_N+\sum_{\emptyset\neq J\subset\{1,2,\dots,N-1\}}
(f_J+d_J(g)t_N)\prod_{j\in J}t_j\right), 
\]
where $d_{\{j_1,j_2,\dots,j_k\}}:=d_{j_1}\circ\dots\circ d_{j_k}$. 
When substituting $0$ into $t_i$ for some $i\neq N$ we have 
\[
\exp2\pi i\left(gt_N+\sum_{i\notin J}
(f_J+d_J(g)t_N)\prod_{j\in J}t_j\right), 
\]
and when substituting $1$ into $t_i$ we have 
\begin{align*}
& \exp2\pi i\left(gt_N+f_i+d_i(g)t_N+\sum_{i\notin J}
\left(f_J+d_J(g)t_N+d_i(f_J)+d_i(d_J(g)t_N)\right)
\prod_{j\in J}t_j\right) \\
&=\exp2\pi i\left(e_i.gt_N+\sum_{i\notin J}
(e_i.f_J+e_i.d_J(g)t_N)\prod_{j\in J}t_j\right) \\
&=\exp2\pi i\left(e_i.\left(gt_N+\sum_{i\notin J}
(f_J+d_J(g)t_N)\prod_{j\in J}t_j\right)\right). 
\end{align*}
Therefore, for any $t_N\in[0,1]$, 
$w(\cdot,\dots,\cdot,t_N)$ belongs to $\cM(\Z^{N-1}\curvearrowright X)$. 
At $t_N=0$, 
\[
w(t_1,t_2,\dots,t_{N-1},0)=u_\eta(t_1,t_2,\dots,t_{N-1}). 
\]
Moreover, 
\begin{align*}
w(t_1,t_2,\dots,t_{N-1},1)
&=\exp2\pi i\left(g+\sum_{\emptyset\neq J\subset\{1,2,\dots,N-1\}}
(f_J+d_J(g))\prod_{j\in J}t_j\right) \\
&=\exp2\pi i\left(\sum_{\emptyset\neq J\subset\{1,2,\dots,N-1\}}
(f_J+d_J(g))\prod_{j\in J}t_j\right) \\
&=\bar\theta(w(t_1,t_2,\dots,t_{N-1},0)), 
\end{align*}
because $g$ is integer-valued and 
\[
f_J+d_J(g)
=(d_{j_1}\circ\dots\circ d_{j_{k-1}})(f_{j_k}+d_{j_k}(g))
=(d_{j_1}\circ\dots\circ d_{j_{k-1}})(f_{j_k}\circ\theta^{-1})
=f_J\circ\theta^{-1}
\]
for any $J=\{j_1,j_2,\dots,j_k\}$. 
It follows that $w$ is in $\cM(\Z^N\curvearrowright X)$ and 
$q(w)=u_\eta$. 
\end{proof}

\begin{theorem}\label{GL5}
If $\cG$ arises from a free action of $\Z^5$ on a Cantor set $X$, 
\[
\Ima D_\cG\subset\frac{1}{2}\Ima(D_\cG\circ\mu_0). 
\]
Thus GL holds for $\cG$ up to a factor of two. 
\end{theorem}

\begin{proof}
We apply the argument above to the case $N=5$. 
Under the identification of 
$K_1(C^*_r(\cH))$ and $K_1(\cM(\Z^4\curvearrowright X))$, 
the map $\mu_1\oplus\tilde\mu^1$ is an isomorphism 
by Theorem \ref{Hirsch4a+} (2). 
For any $h\in H^1(\cH)$ with $h=H^1(\theta)(h)$, 
the lemma above tells us that 
there exists a unitary $w\in\cM(\Z^5\curvearrowright X)$ 
such that $K_1(q(w))=\tilde\mu^1(h)$. 
As in the proof of Lemma \ref{Connes}, 
we have $\langle[\hat\tau_\nu],K_1(w)\rangle=0$ 
for any $\nu\in M(\cG)$. 
This, together with Proposition \ref{halfGL} and Theorem \ref{GL4}, 
implies 
\[
\Ima D_\cG\subset\frac{1}{2}\Ima(D_\cG\circ\mu_0)+\Ima(D_\cG\circ\iota)
\subset\frac{1}{2}\Ima(D_\cG\circ\mu_0), 
\]
as desired. 
\end{proof}

\begin{remark}
We compare Theorem \ref{GL5} 
with the denominator appearing in the Chern character approach. 
In the framework of tiling spaces considered 
in \cite[Remark 9.6]{ADRS21JGP}, 
it is observed that, in dimension $d$, 
the top-degree part of the Chern character 
\[
\operatorname{ch}_d:K^{-d}(\Omega_T)\to\check H^d(\Omega_T;\Q) 
\]
does not necessarily take values in $\check H^d(\Omega_T;\Z)$, 
but it does after multiplication by $\lceil d/2\rceil!$. 
Thus, in that approach, 
the gap-labelling statement is obtained only up to this factor 
in higher dimensions. 
In particular, 
the corresponding factor is $2$ in dimension $4$, 
and $3!=6$ in dimension $5$. 
Theorem \ref{GL5} gives, for free $\Z^5$-systems, 
the stronger inclusion
\[
\Ima D_\cG\subset\frac{1}{2}\Ima(D_\cG\circ\mu_0). 
\]
The factor $1/2$ in our result has a different origin: 
it comes from the de la Harpe--Skandalis determinant and 
the transposition decomposition in topological full groups, 
rather than from the denominators in the Chern character. 
\end{remark}

\newcommand{\noopsort}[1]{}
\providecommand{\bysame}{\leavevmode\hbox to3em{\hrulefill}\thinspace}
\providecommand{\MR}{\relax\ifhmode\unskip\space\fi MR }
\providecommand{\MRhref}[2]{%
  \href{http://www.ams.org/mathscinet-getitem?mr=#1}{#2}
}
\providecommand{\href}[2]{#2}

\end{document}